\documentclass[12pt,a4paper]{article}%
\usepackage{amsmath}
\usepackage{graphicx}
\usepackage{epsfig}%
\usepackage{amsfonts}%
\usepackage{amssymb}

\usepackage{color}
\newcommand{\cn}{\color{black}}

\usepackage[normalem]{ulem}

\newtheorem{theorem}{Theorem}[section]
\newtheorem{corollary}[theorem]{Corollary}

\newtheorem{definition}[theorem]{Definition}

\newtheorem{lemma}[theorem]{Lemma}
\newtheorem{proposition}[theorem]{Proposition}
\newtheorem{remark}[theorem]{Remark}

\newenvironment{proof}[1][Proof]{\noindent \emph{#1.} }{\hfill \ 
\rule{0.5em}{0.5em}}

\makeatletter\@addtoreset{equation}{section}\makeatother
\makeatletter\@addtoreset{figure}{section}\makeatother
\makeatletter\@addtoreset{table}{section}\makeatother
\begin{document}
 
\title{Reduced Higher Order SVD: ubiquitous rank-reduction \\   
 method in tensor-based scientific computing}
%  \title{Ubiquitous Nature of the Reduced Higher Order SVD  \\
%  in   Tensor-based Scientific Computing }

%   \title{New horizons of the reduced higher-order SVD \\ in 
%   large-scale tensor-based scientific computing and \\ data modeling}
%  
 
\author{Venera Khoromskaia \thanks{Max-Planck-Institute for
        Mathematics in the Sciences, Inselstr.~22-26, D-04103, Leipzig, Germany   ({\tt vekh@mis.mpg.de}).}
        \and
        Boris N. Khoromskij \thanks{Max-Planck-Institute for
        Mathematics in the Sciences, Inselstr.~22-26, D-04103 Leipzig,
        Germany ({\tt bokh@mis.mpg.de}).}
        }
 
  \date{}

\maketitle

\begin{abstract}
Tensor numerical methods, 
 based on the rank-structured tensor representation of $d$-variate functions and operators 
 discretized on large $n^{\otimes d }$ grids,  
are designed to provide $O(dn)$ complexity of numerical calculations  contrary to $O(n^d)$ 
scaling  by conventional   grid-based  methods.  
 However,  multiple  tensor operations may lead to enormous increase 
 in the  tensor ranks (curse of ranks) of the  target  data, making calculation intractable. 
 Therefore  one of the most important steps in tensor calculations is the robust and efficient 
 rank  reduction procedure   which should be performed many times in the course of 
various tensor transforms in multidimensional operator and function calculus. 
The rank reduction scheme based on the Reduced Higher Order SVD (RHOSVD) 
 introduced in \cite{KhKh3:08} played a significant role in the development of tensor numerical methods.
Here, we  briefly survey the essentials of RHOSVD method and then focus on  some new theoretical 
and computational aspects 
of the RHOSVD  demonstrating  that this rank reduction 
technique constitutes the  basic  ingredient in tensor computations for real-life problems. 
In particular,  the stability analysis of RHOSVD is presented.  
We   recall  the performance of the RHOSVD in tensor-based calculation of the Hartree potential 
in computational quantum chemistry.
  We introduce the multilinear algebra of tensors represented in the range-separated (RS) tensor format. 
  This allows to apply the 
RHOSVD rank-reduction techniques to non-regular functional data with many singularities, for example, 
to the rank-structured computation of the collective multi-particle interaction potentials in 
bio-molecular modeling, as well as to complicated composite radial functions.   
The new  theoretical and numerical results  on application of the RHOSVD in  scattered data modeling are presented. 
We  finally underline that RHOSVD proved to be the efficient rank reduction technique in 
numerous applications ranging from numerical treatment of multi-particle systems in material 
sciences up to a numerical solution of PDE constrained control
problems in $\mathbb{R}^d$.  
  \end{abstract}



  \section{Introduction} \label{sec:Introd}  
    
  The mathematical models in large-scale scientific computing are often described by steady state or dynamical PDEs.
  The underlying physical, chemical or biological systems usually live in 3D physical space $\mathbb{R}^3$ 
  and may depend on many 
  structural parameters. The solution of arising discrete systems of equation and optimization 
  of the model parameters 
  leads to the challenging numerical problems. Indeed, the accurate grid-based approximation of operators and 
  functions involved 
  requires large spacial grids in $\mathbb{R}^d$, leading to the considerable storage space and 
  to the implementation of 
  various algebraic operations on huge vectors and matrices. 
  For further discussion we shall assume that all functional entities are discretized on $n^{\otimes d}$ spacial grids
  where the univariate  grid size $n$ may vary in the range of several thousands. 
  The linear algebra on $N$-vectors and $N\times N$ matrices with $N=n^d$
  quickly becomes non-tractable as $n$ and $d$ increase.

 Tensor numerical methods \cite{Khor_book:18,Khor_bookQC:18} provide means to   overcome the problem of 
 the exponential increase of numerical complexity in the dimension of the 
 problem $d$, 
  due to their intrinsic feature of reducing the computational costs of multilinear algebra on rank-structured 
  data to merely linear scaling in both the grid-size $n$ and dimension $d$. They appeared as bridging  of the 
  algebraic tensor decompositions initiated in chemometrics 
\cite{Comon:02,Comon:09,smilde-book-2004,Cichocki:2016,DMV-SIAM2:00,Sidi:02,Sidi:17} and of the nonlinear 
approximation theory 
on separable low-rank representation of multivariate functions and 
operators \cite{HaKhtens:04I,KhKh:06}. 
    The canonical \cite{Hitch:27}, Tucker \cite{Tuck:1966}, tensor train (TT) and hierarchical Tucker (HT) formats 
    \cite{OsTy_TT:09,OsTy_TT2:09,HaKu:09} are the most
  commonly used rank-structured parametrizations in applications of modern tensor numerical methods. 
  Further data-compression to the logarithmic scale can be achieved by using the quantized-TT  
  (QTT) \cite{KhQuant:09} tensor approximation.
    At present there is an active research toward further 
  progress of tensor numerical methods in scientific computing
  \cite{Khor_book:18,Khor_bookQC:18,KSU:13,LiKeKKMa:19,RaOs_Conv:15}. 
  In particular, there are considerable achievements of tensor-based approaches in computational chemistry 
  \cite{VeKh_Diss:10,KhKh_CPC:13,KKF:09,DoKhOsel:11,KKNSchw:13}, 
  in bio-molecular modeling 
  \cite{BKK_RS:17,BeKhKhKwSt:18,BKhor_Dirac:18,KKKSB:21},
  in optimal control problems (including the case of fractional control)
  \cite{HKKS:18,DKK:21,DP:19,SKKS:21}, and in many other fields  
  \cite{BSU:16,BEEN:19,MaEsLiHaZa:19,LRSV:12,Cichocki:2016,Litv:22}.

  Here we notice that tensor numerical methods proved to be efficient when all input data and 
  all intermediate quantities 
  within the chosen computational scheme are presented in certain 
  low-rank tensor format with controllable rank parameters,
  i.e. on low-rank tensor manifolds. In turn, tensor decomposition of the full format data arrays is considered 
  as N-P hard problem.
  For example, the truncated HOSVD \cite{DMV-SIAM2:00} of an $n^{\otimes d}$-tensor in the Tucker format 
  amounts to $O(n^{d+1})$   arithmetic operations while 
  the respective cost of the TT and HT higher-order SVD \cite{Osel_TT:11,Gras:10}  is estimated by $O(n^{\frac{3}{2} d})$, 
  indicating that rank decomposition   of full format tensors
  still suffers from the ``curse of dimensionality'' and practically could not be applied 
  in large scale computations. 
  
  On the other hand, often, the initial data for complicated numerical algorithms may be chosen 
  in the canonical/Tucker tensor formats, 
  say as a result of discretization of a short sum of Gaussians or multivariate polynomials, 
  or as a result of the analytical approximation by using Laplace
  transform representation and sinc-quadratures \cite{Khor_book:18}.
  However, the ranks of tensors are multiplied in the course of various tensor operations, leading to dramatical
  increase in the rank parameter (``curse of ranks'') of a resulting tensor, thus making 
  tensor-structured calculation intractable. 
  Therefore, a stable rank reduction schemes are the main  prerequisite for 
  the success of rank-structured tensor techniques.

  Invention of the Reduced Higher Order SVD (RHOSVD) 
  in \cite{KhKh3:08} and the corresponding rank reduction procedure based on the 
  canonical-to-Tucker transform and subsequent  canonical approximation of the small Tucker core 
  (Tucker-to-canonical transform) 
  was a decisive step in development of the tensor numerical methods in scientific computing. 
  In contrast to the conventional HOSVD, the RHOSVD does not need a construction of the full 
  size tensor for finding  
  the orthogonal subspaces of the Tucker tensor representation. 
  Instead, RHOSVD applies to numerical   data  in the canonical tensor format (with possibly large initial rank $R$) 
  and exhibits the $O(d nR \min\{n,R\})$ complexity, uniformly in the dimensionality of the problem, $d$,
    and it was an essential step ahead in development of tensor-structured numerical techniques.     
  
  In particular, this rank reduction scheme was applied to calculation of 3D and 6D
  convolution integrals in tensor-based solution of the Hartree-Fock equation
  \cite{KhKh3:08,VeKh_Diss:10}. Combined with the Tucker-to-canonical transform, 
  this algorithm provides a stable procedure for the rank reduction of possibly 
  huge ranks in tensor-structured calculations of the Hartree potential. 
  The RHOSVD based rank reduction scheme for the canonical tensors  is specifically useful for 3D problems, 
  which are most often in real-life applications.
  However, the RHOSVD-like  procedure can be also efficiently applied in the construction of 
  the TT tensor format from the canonical tensor input, which often appears in tensor calculations.

 The RHOSVD is the basic tool for the construction of the range-separated (RS) tensor 
 format introduced in \cite{BKK_RS:17} for the low-rank 
  tensor representation of the bio-molecular long-range electrostatic   potentials. 
  Recent example on the RS representation of the multi-centered Dirac delta function 
  \cite{BKhor_Dirac:18} paves the way for efficient solution decomposition scheme introduced  for the 
  Poisson-Boltzmann equation   \cite{BeKhKhKwSt:18,KKKSB:21}.

  In some applications the data  could be presented as a sum of highly localized and rank-structured 
  components so that their further numerical treatment again requires the rank reduction procedure 
  (see section \ref{sec:Apll2Data} concerning the long-range potential calculation for many-particle system).   
  Here, we consider the example on representation of the scattered data by a sum of Slater kernels 
  and show the existence of the low-rank representation for such data in the RS tensor format.
  The numerical examples demonstrate the practical efficiency of such kind of tensor interpolation.

   Rank reduction procedure by using the RHOSVD is a mandatory part in   solving the three-dimensional 
   elliptic and pseudo-differential equations in the rank-structured tensor format. 
%    The convenient tensor format for these problems is a canonical tensor representation of both the governing operator, 
%    and of the initial guess as well as of the right hand side. 
   In the course of preconditioned iterations, 
   the tensor ranks of the governing operator, the precoditioner  and of the current iterand are multiplied
   at each iterative step, and 
   therefore a fast and robust rank reduction techniques is the prerequisite for 
   such methodology applied in the framework of iterative elliptic problem solvers.
   In particular, this approach was applied to the PDE constrained (including the case of fractional operators) 
   optimal control problems \cite{HKKS:18,SKKS:21}. 
   As result, the computational complexity can be reduced to 
   almost linear scale, $O(nR)$, contrary to conventional $O(n^3)$ complexity, 
   as demonstrated by numerics in \cite{HKKS:18,SKKS:21}.

  Tensor-based algorithms and methods are now being widely used and developed further
  in the communities of  scientific computing and data science. Tensor techniques  
  evolve   in traditional tensor decompositions in data processing   \cite{smilde-book-2004,MaEsLiHaZa:19,EGLS:21},
  and they are actively promoted for  tensor-based solution of the multidimensional problems  
  in numerical analysis and quantum chemistry
  \cite{KKF:09,RaOs_Conv:15,SKKS:21,Khor_book:18,KKF:09,KU:16,DP:19,ORU:18}.
  Notice that in the case of higher dimensions the rank reduction in the canonical format can be performed 
 directly (i.e., without intermediate use of the Tucker approximation) by using 
 the cascading ALS iteration in the CP format, see \cite{KhSch:11} concerning the tensor-structured 
 solution of the stochastic/parametric PDEs.

   The rest of the papers is organized as follows. In section \ref{sec:RHOSDV_all} 
   we sketch some results on the construction of the RHOSVD and present
   some old and new results on the stability of error bounds. In section
   \ref{ssec:MixedTuckCp} we discuss the mixed canonical-Tucker tensor format and 
   the Tucker-to-canonical transform. Section \ref{ssec:Multi_Int}
   recalls the results from \cite{VeKh_Diss:10} on calculation of the 
   multidimensional convolution integrals with the Newton kernel arising in computational quantum chemistry.
   Section \ref{sec:RHOSVD_RSf} addresses the application of RHOSVD to RS parametrized tensors.
   In \S\ref{ssec:RS_form} we discuss  the application of RHOSVD in  multilinear operations 
   of data in the RS tensor format. The scattered data modeling is considered in \S\ref{sec:Apll2Data}
   from both theoretical and computational aspects. 
   Application of RHOSVD for tensor-based representation of Greens kernels is discussed in 
   section \ref{sec:Fundam_sol}. Section \ref{sec:solvers} gives a short sketch of RHOSVD in application 
   to tensor-structured elliptic problem solvers.

\section{Reduced HOSVD and CP-to-Tucker Transform}
 \label{sec:RHOSDV_all}

\subsection{Reduced HOSVD: error bounds}\label{ssec:RHOSDV_theory}

 In computational schemes including bilinear tensor-tensor or matrix-tensor operations
the increase of tensor ranks leads to the critical loss of efficiency. 
Moreover, in many applications, for example in electronic structure calculations,
the canonical tensors with large rank parameters arise as the result of polynomial type or convolution transforms
of some function related tensors (say, electron density, the Hartree potential, etc.).
In what follows, we present the new look on the direct method of rank reduction for the canonical tensors
with large initial rank, the reduced HOSVD, first introduced and analyzed in \cite{KhKh3:08}.

 We denote by $\mbox{\boldmath{$\mathcal{T}$}}_{{\bf r},{\bf n}}$ the class of tensors  ${\bf A}\in \mathbb{R}^{n\otimes d}$ 
 parametrized in the rank-${\bf r}$, ${\bf r}=(r_1,\cdots,r_d)$ orthogonal Tucker format,
$$
{\bf A}={\boldsymbol{\beta}}\times_{1}V^{(1)}\times_{2}
\cdots   \times_{d} V^{(d)} \in \mbox{\boldmath{$\mathcal{T}$}}_{{\bf r},{\bf n}},
$$
 with the orthogonal side-matrices
 $V^{(\ell)}=[{\bf v}^{(\ell)}_1 \ldots {\bf v}^{(\ell)}_{r_\ell} ]
\in \mathbb{R}^{ n \times  r_\ell }$  and  with the core coefficient tensor
${\boldsymbol{\beta}} \in \mathbb{R}^{r_1\times...\times r_d}$.

Likewise, $\mbox{\boldmath{$\mathcal{C}$}}_{{R},{\bf n}}$ denotes the class of rank-$R$ canonical tensors. 
For given ${\bf A}\in\mbox{\boldmath{$\mathcal{C}$}}_{{R},{\bf n}} $
in the rank-$R$ canonical format, 
\begin{equation}\label{eqn:can_A3}
  {\bf A} ={\sum}_{\nu =1}^{R} \xi_{\nu}
  {\bf u}^{(1)}_{\nu}  \otimes \ldots \otimes {\bf u}^{(d)}_{\nu},
 \quad  \xi_{\nu}\in \mathbb{R},
\end{equation}
with normalized canonical vectors, i.e. $\|{\bf u}^{(\ell)}_\nu\|=1$ for $\ell=1,...,d$, $\nu=1,...,R$.

The standard algorithm for the Tucker tensor decomposition  \cite{DMV-SIAM2:00} is based 
on HOSVD applied to full tensors of size $n^d$ which exhibits $O(n^{d+1})$ computational
complexity. The question is how to simplify the HOSVD Tucker approximation in the case of 
canonical input tensor in the form (\ref{eqn:can_A3}) without use of 
 the full format representation of ${\bf A}$,
 and in the situation when the CP rank parameter $R$ and the mode sizes $n$ of the input can be 
 sufficiently large. %  ${\bf A}$.
 
 First, let us use the equivalent (nonorthogonal) rank-${\bf r}=(R,...,R)$ Tucker 
 representation of the tensor (\ref{eqn:can_A3}), 
\begin{equation}\label{eqn:can_A_Tuck}
 {\bf A}= \boldsymbol{\xi} \times_1 {U}^{(1)}\times_2 {U}^{(2)} \cdots \times_d {U}^{(d)}, \quad
\boldsymbol{\xi}=diag\{\xi_{1},...,\xi_{R}\},
\end{equation}
via contraction of the diagonal tensor $\boldsymbol{\xi}=diag\{\xi_{1},...,\xi_{R}\}\in \mathbb{R}^{R\otimes d}$ 
with $\ell$-mode side matrices 
$U^{(\ell)}=[{\bf u}^{(\ell)}_{1}, ..., {\bf u}^{(\ell)}_{R} ]\in \mathbb{R}^{n\times R}$, 
 see Figure \ref{fig:C2T1}. 
\begin{figure}[tbh]
  \begin{center}
 \includegraphics[width=15cm]{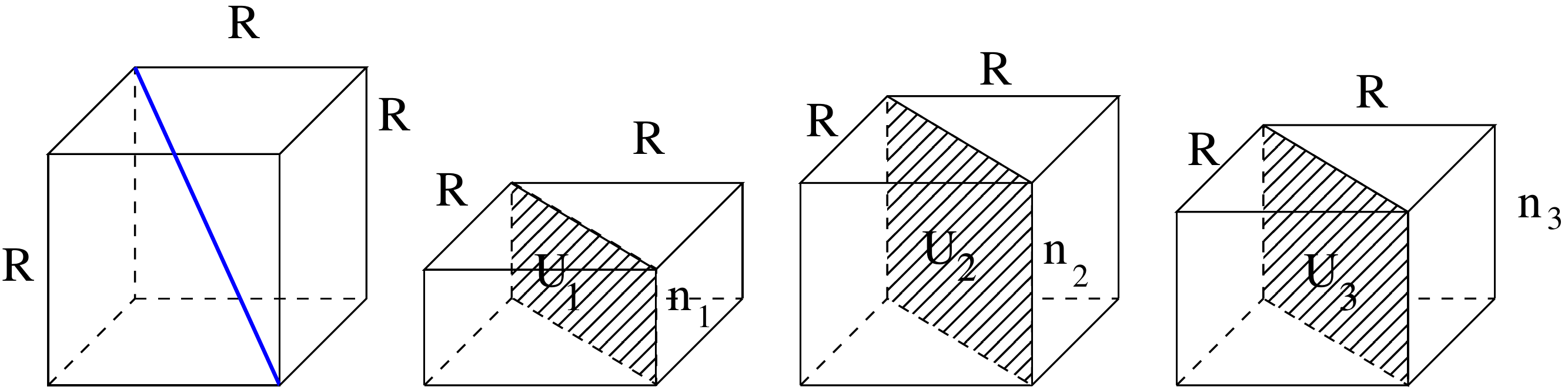}
\end{center}
\caption{\small Contracted product representation of the rank-$R$ canonical tensor.}    
\label{fig:C2T1}%
\end{figure} 
Then the  problem of canonical to Tucker approximation can be solved by the method of 
reduced HOSVD (RHOSVD) introduced in \cite{KhKh3:08}.
 The basic idea of the reduced HOSVD is that for large (function related)
tensors given in the canonical format their HOSVD does not require 
the construction of a tensor in the full format and SVD based computation of its matrix unfolding. 
Instead, it is sufficient to compute the SVD of the directional matrices $U^{(\ell)}$ 
in (\ref{eqn:can_A_Tuck}) composed by
only the vectors of the canonical tensor in every dimension separately,
as shown in Figure \ref{fig:C2T1}. This will
provide the initial guess for the Tucker orthogonal basis in the given dimension. 
For the practical applicability, the results of the approximation theory 
on the low-rank approximation to the 
multivariate functions, exhibiting exponential error decay in the Tucker rank, 
are of the principal significance \cite{Khor1:06}.
 
In what follows, we suppose that $n\leq R$ and denote the SVD of the  side-matrix $U^{(\ell)}$ by
\begin{equation}\label{eqn:SVD_Ul}
U^{(\ell)}={Z}^{(\ell)} D_\ell {V^{(\ell)}}^T=
\sum\limits_{k=1}^n \sigma_{\ell,k} {\bf z}_k^{(\ell)}\; {{\bf v}_k^{(\ell)}}^T,\quad
{\bf z}_k^{(\ell)}\in \mathbb{R}^{n},\; {\bf v}_k^{(\ell)}\in \mathbb{R}^{R},
\end{equation}
with the orthogonal matrices 
${ Z}^{(\ell)}=[{\bf z}_1^{(\ell)},...,{\bf z}_n^{(\ell)}]\in \mathbb{R}^{n\times n}$, and
${V}^{(\ell)}=[{\bf v}_1^{(\ell)},...,{\bf v}_n^{(\ell)}]\in \mathbb{R}^{R\times n}$, $\ell=1,...,d$.
We use the following notations for the vector entries, $v_k^{(\ell)}(\nu)=v_{k,\nu}^{(\ell)}$
($\nu=1,...,R$).

To fix the idea, we introduce the vector of rank parameters, ${\bf r}=(r_1,...,r_d)$, and let  
\begin{equation}\label{eq:svdr}
 U^{(\ell)}\approx W^{(\ell)}:= {Z}_0^{(\ell)} D_{\ell,0} {V_0^{(\ell)}}^T,
\end{equation}
be the rank-$r_\ell$ truncated SVD of the side-matrix $U^{(\ell)}$ ($\ell=1,...,d$). Here the matrix
$ D_{\ell,0}=\mbox{diag} \{\sigma_{\ell,1},\sigma_{\ell,2},...,\sigma_{\ell,r_\ell}\}$  
is the submatrix of $D_\ell$ in (\ref{eqn:SVD_Ul}) and
$$
%D_{\ell,0}=\mbox{diag} \{\sigma_{\ell,1},\sigma_{\ell,2},...,\sigma_{\ell,r_\ell}\}\quad \mbox{and}\quad 
{Z}_0^{(\ell)}=[{\bf z}^{(\ell)}_1,...,{\bf z}^{(\ell)}_{r_\ell}] \in \mathbb{R}^{n\times r_\ell},\;
{V_0}^{(\ell)}\in \mathbb{R}^{R\times r_\ell}, 
$$
 represent the respective dominating $(n\times r_\ell)$-submatrices  of the left and right factors
in the complete SVD decomposition in (\ref{eqn:SVD_Ul}). 
 
\begin{definition}\label{def:7.2} (Reduced HOSVD, \cite{KhKh3:08}).  
Given the canonical tensor ${\bf A}\in\mbox{\boldmath{$\mathcal{C}$}}_{{R},{\bf n}} $, the
truncation rank parameter ${\bf r}$, ($r_\ell\leq R$), and rank-$r_\ell$ truncated SVD of $U^{(\ell)}$, 
see (\ref{eq:svdr}), then
the  RHOSVD approximation of ${\bf A}$ is defined by the rank-${\bf r}$ orthogonal Tucker tensor
%\begin{equation}
\begin{eqnarray} \label{eqn:RHOSVD}
\nonumber {\bf A}_{({\bf r})}^0 &:=& \boldsymbol{\xi} \times_1 W^{(1)}\times_2 \cdots \times_d W^{(d)}=
\boldsymbol{\xi} \times_1 \left[{Z}_0^{(1)} D_{1,0} {V_0^{(1)}}^T \right] \times_2 
%\left[{Z}_0^{(2)}D_{2,0} {W_0^{(2)}}^T\right] 
\cdots \times_d \left[{Z}_0^{(d)}D_{d,0} {V_0^{(d)}}^T \right] \\
&=& \left(\boldsymbol{\xi} \times_1 [D_{1,0} {V_0^{(1)}}^T ] \times_2 \cdots \times_d [D_{d,0} {V_0^{(d)}}^T]\right)
\times_1 {Z}_0^{(1)} \times_2 \cdots \times_d {Z}_0^{(d)} \,\, \in {\cal T}_{\bf r},
\end{eqnarray}
%\end{equation}
obtained  by the projection of canonical side matrices $U^{(\ell)}$ onto the  left orthogonal singular 
matrices $Z_0^{(\ell)}$, defined in (\ref{eq:svdr}).
\end{definition}

The sub-optimal Tucker approximant  (\ref{eqn:RHOSVD}) is simple to 
compute and it provides accurate approximation to the initial canonical tensor even with rather small Tucker rank.
Moreover, this provides the good initial guess to calculate the best rank-${\bf r}$ Tucker approximation 
by using the ALS iteration.
In our numerical practice, usually, only one or two ALS iterations are required for convergence.
For example, in case $d=3$,
algorithmically, the one step of the canonical-to-Tucker ALS algorithm reduces to the following operations. 
Redistributing the normalization coefficients $\boldsymbol{\xi}$ 
and substituting (\ref{eqn:RHOSVD}) in the equation  (\ref{eqn:can_A_Tuck}), we obtain  
\begin{equation}\label{eqn:can_A_Tuck_als}
 {\bf A}= {Z}_0^{(1)}  \times_1 %D_{1,0} {V_0^{(1)}}^T \times_1 
 {\bf A}_2 \times_2  {Z}_0^{(3)}   ,  %D_{3,0} {V_0^{(3)}}^T,
%  
%  \boldsymbol{\xi} \times_1 {U}^{(1)}\times_2 {U}^{(2)} \cdots \times_d {U}^{(d)}, \quad
% \boldsymbol{\xi}=diag\{\xi_{1},...,\xi_{R}\},
\end{equation}
where ${\bf A}_2$ is the contraction 
\[
 {\bf A}_2= \boldsymbol{\xi} \times_1 D_{1,0} {V_0^{(1)}}^T
 \times_3 D_{3,0} {V_0^{(3)}},
\]
as shown in Figure \ref{figs1:Can_Tuck}, and then we optimize the orthogonal subspace in the second variable
by calculating the best rank-$r_2$ approximation to the $r_1 r_3 \times n_2$ matrix 
unfolding to the tensor ${\bf A}_2$. 
The similar contracted 
product representation can be used when $d>3 $, as well as for the construction of the TT representation
for the canonical input.
\begin{figure}[h]
 \centering
\includegraphics[width=10cm]{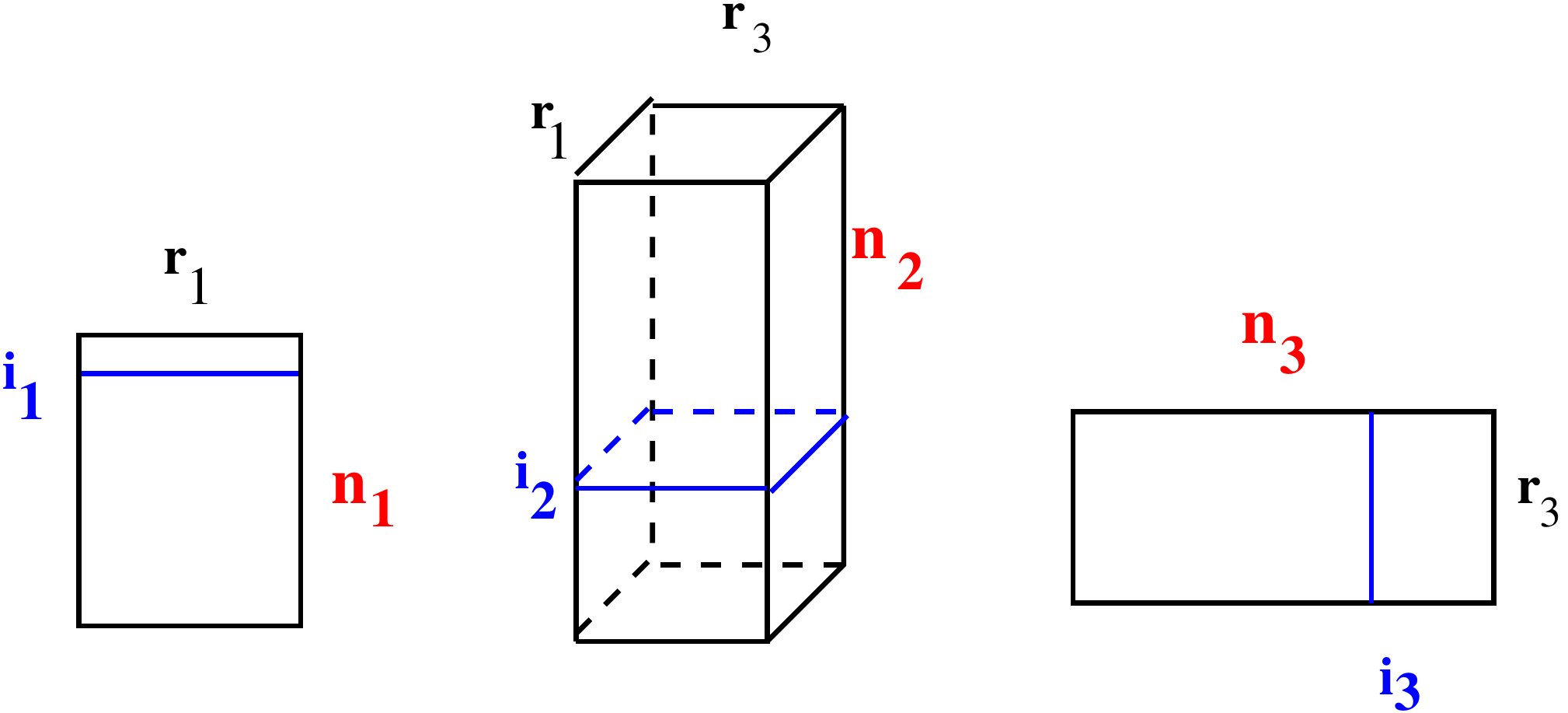} 
\caption{A first step in canonical-to-Tucker decomposition.}
\label{figs1:Can_Tuck}
\end{figure}

Here we notice that the core tensor in the RHOSVD decomposition can be represented in the CP data-sparse format. 
\begin{proposition} \label{prop:RHOSVD_mixTuck}
The core tensor 
\[
  {\boldsymbol{\beta}}_0 = \boldsymbol{\xi} \times_1 [D_{1,0} {V_0^{(1)}}^T ] \times_2 \cdots 
  \times_d [D_{d,0} {V_0^{(d)}}^T]\in \mathbb{R}^{r_1 \times \cdots \times r_d},  
\] 
in the orthogonal Tucker representation (\ref{eqn:RHOSVD}),
$
 {\bf A}_{({\bf r})}^0 ={\boldsymbol{\beta}}_0\times_{1}Z^{(1)}_0 \times_{2}
\cdots   \times_{d} Z^{(d)}_0 \in \mbox{\boldmath{$\mathcal{T}$}}_{{\bf r},{\bf n}},
$
% with 
% \[
%   {\boldsymbol{\beta}}_0 = \boldsymbol{\xi} \times_1 [D_{1,0} {V_0^{(1)}}^T ] \times_2 \cdots \times_d [D_{d,0} {V_0^{(d)}}^T],  
% \]
can be recognized as the rank-$R$ canonical tensor of size $r_1 \times \cdots \times r_d$ with the storage request 
$R(\sum\limits_{\ell=1}^d r_\ell)$, which can be calculated entry-wise in $O(R r_1 \cdots r_d)$ operations. 
\end{proposition}
Indeed, introducing the matrices 
${U_0^{(\ell)}}^T = D_{\ell,0} {V_0^{(\ell)}}^T\in \mathbb{R}^{r_\ell \times R}$,
for $\ell=1,\ldots d$, we conclude that the canonical core tensor ${\boldsymbol{\beta}}_0$ is determined by the 
$\ell$-mode side matrices ${U_0^{(\ell)}}^T$. In the other words, the tensor ${\bf A}_{({\bf r})}^0$ 
is represented in the mixed Tucker-canonical format getting rid of the ``curse of dimensionality'', 
see also \S\ref{ssec:MixedTuckCp} below.

The accuracy of the RHOSVD approximation can be controlled by the given 
$\varepsilon$-threshold in truncated SVD of side matrices $U^{(\ell)}$.
The following theorem proves the absolute error bound for the RHOSVD approximation. 
%can be controlled by the rank truncation errors of the truncated SVD for the side matrices $U^{(\ell)}$.
\begin{theorem} \label{thm:7.3} (RHOSVD error bound, \cite{KhKh3:08}). 
For given ${\bf A}\in  \mbox{\boldmath{$\mathcal{C}$}}_{{ R},{\bf n}}$ 
in (\ref{eqn:can_A3}), 
let $\sigma_{\ell,1}\geq\sigma_{\ell,2} ...  \geq\sigma_{\ell,\min(n,R)}$
 be the  singular values of $\ell$-mode side matrices
$U^{(\ell)}\in \mathbb{R}^{n\times R}$ ($\ell=1,...,d$) with
normalized skeleton vectors. Then the error of RHOSVD approximation,
${\bf A}_{({\bf r})}^0$, is bounded by 
 \begin{equation}\label{error bound CtoT}
\|{\bf A}- {\bf A}_{({\bf r})}^0\| \leq \|\boldsymbol{\xi}\| \sum\limits_{\ell=1}^d
(\sum\limits_{k=r_\ell +1}^{\min(n,R)} \sigma_{\ell,k}^2)^{1/2},\quad 
\|\boldsymbol{\xi}\| = \sqrt{\sum\limits_{\nu =1}^{R} \xi_{\nu}^2}.
\end{equation}
\end{theorem}
\begin{proof} 
Using the contracted product representations of 
${\bf A}\in \mbox{\boldmath{$\mathcal{C}$}}_{{R},{\bf n}}$ and 
${\bf A}_{({\bf r})}^0 \in {\cal T}_{\bf r}$, 
and introducing the $\ell$-mode residual
\[
\Delta^{(\ell)}={U}^{(\ell)}-{Z}_0^{(\ell)} D_{\ell,0} {V_0^{(\ell)}}^T,\quad
\{\Delta^{(\ell)}\}_\nu=
 \sum\limits_{k=r_\ell +1}^n\sigma_{\ell,k} {\bf z}^{(\ell)}_k v^\ell_{k,\nu},\quad \nu=1,...,R,
\]
with notations
$
{V_0^{(\ell)}}=[{\bf v}^{(\ell)}_1,...,{\bf v}^{(\ell)}_{r_\ell}]^T, \quad 
{\bf v}^{(\ell)}_k=\{v^{\ell}_{k,\nu}\}_{\nu=1}^R\in \mathbb{R}^R,
$
we arrive at the following expansion for the approximation error in the form
% $\| {\bf A} - {\bf A}_{({\bf r})}^0 \|$ in the form 
\[
 {\bf A} - {\bf A}_{({\bf r})}^0 =   
  \boldsymbol{\xi} \times_1 U^{(1)}\times_2 \cdots \times_d U^{(d)} - 
  \boldsymbol{\xi} \times_1 W^{(1)}\times_2 \cdots \times_d W^{(d)}
:= \sum\limits_{\ell=1}^d {\bf B}_\ell, 
\]
where 
\begin{eqnarray*}
{\bf B}_\ell &=& \boldsymbol{\xi} \times_1 {U}^{(1)}\cdots
\times_{\ell-1} U^{(\ell-1)}\times_\ell \Delta^{(\ell)}\times_{\ell+1} W^{(\ell+1)}\cdots \times_d W^{(d)}\\
 &=& \sum\limits_{\nu=1}^R {\xi}_\nu \left[ {\bf u}^{(\ell)}_\nu  \cdots 
\times_{\ell-1}{\bf u}^{(\ell-1)}_\nu 
 \times_{\ell} \{\Delta^{(\ell)}\}_\nu
 \times_{\ell+1} \sum\limits_{k=1}^{r_{\ell+1}}\sigma_{\ell+1,k}
 {\bf z}^{(\ell+1)}_k v^{\ell+1}_{k,\nu}   \cdots 
\times_d \sum\limits_{k=1}^{r_d}\sigma_{d,k} {\bf z}^{(d)}_k v^{d}_{k,\nu}\right].
\end{eqnarray*}
This leads to the error bound (by the triangle inequality)
$$
 \| {\bf A} - {\bf A}_{({\bf r})}^0 \|  \leq  \sum\limits_{\ell=1}^d \|{\bf B}_\ell \|,
$$
% where the $\ell$th term ${\bf B}_\ell$ can be represented by using vector notations,
% \[
% \sum\limits_{\nu=1}^R {\xi}_\nu \left[ {\bf u}^{(\ell)}_\nu  \cdots 
% \times_{\ell-1}{\bf u}^{(\ell-1)}_\nu 
%  \times_{\ell} \{\Delta^{(\ell)}\}_\nu
%  \times_{\ell+1} \sum\limits_{k=1}^{r_{\ell+1}}\sigma_{\ell+1,k}
%  {\bf z}^{(\ell+1)}_k v^{\ell+1}_{k,\nu}   \cdots 
% \times_d \sum\limits_{k=1}^{r_d}\sigma_{d,k} {\bf z}^{(d)}_k v^{d}_{k,\nu}\right],
% \]
providing the estimate (in view of $\|{\bf u}^{(\ell)}_\nu \|=1$, $\ell=1,...,d$, $\nu=1,...,R$)
\[
 \Vert {\bf B}_\ell \Vert\leq \sum\limits_{\nu=1}^R |{\xi}_\nu| 
(\sum\limits_{k=r_\ell +1}^n \sigma_{\ell,k}^2 (v^{\ell}_{k,\nu})^2)^{1/2}\,
(\sum\limits_{k=1}^{r_{\ell+1}}\sigma_{\ell+1,k}^2  (v^{\ell+1}_{k,\nu})^2)^{1/2} \cdots
 (\sum\limits_{k=1}^{r_d}\sigma_{d,k}^2  (v^{d}_{k,\nu})^2)^{1/2}.
\] 
Furthermore, since $U^{(\ell)} $ has normalized columns, i.e., 
$
 \|{\bf u}^{(\ell)}_\nu  \|=
\|\sum\limits_{k=1}^{n}\sigma_{\ell,k} {\bf z}^{(\ell)}_k  {v^{\ell}_{k,\nu}} \|=1,\quad \ell=1,...,d,
$
we obtain
$\sum\limits_{k=1}^{n}\sigma_{\ell,k}^2  (v^{\ell}_{k,\nu})^2 =1$ for
$\ell=1,...,d$ $\nu=1,...,R$.
Now the error estimate follows 
\begin{eqnarray*}
 \| {\bf A} - {\bf A}_{({\bf r})}^0 \| 
 & \leq & \sum\limits_{\ell=1}^d \sum\limits_{\nu=1}^R |{\xi}_\nu|
 \left(\sum\limits_{k=r_\ell +1}^n
   \sigma_{\ell,k}^2 (v^{\ell}_{k,\nu})^2\right)^{1/2}\\
 & \leq &   \sum\limits_{\ell=1}^d
 \left(\sum\limits_{\nu=1}^R {\xi}_\nu^2\right)^{1/2} 
\left(\sum\limits_{\nu=1}^R 
\sum\limits_{k=r_\ell +1}^n \sigma_{\ell,k}^2 (v^{\ell}_{k,\nu})^2\right)^{1/2} \\
& = & \sum\limits_{\ell=1}^d\|\boldsymbol{\xi}\|
\left(\sum\limits_{k=r_\ell +1}^n \sigma_{\ell,k}^2 
\sum\limits_{\nu=1}^R (v^{\ell}_{k,\nu})^2\right)^{1/2} 
= \|\boldsymbol{\xi}\| \sum\limits_{\ell=1}^d
\left(\sum\limits_{k=r_\ell +1}^n \sigma_{\ell,k}^2\right)^{1/2}.
\end{eqnarray*}
The case $R < n$ can be analyzed along the same line.
\end{proof}

The error estimate in Theorem \ref{thm:7.3} differs from the case of complete HOSVD 
by the extra factor $\|\boldsymbol{\xi}\|$, which is the payoff for the lack of orthogonality
in the canonical input tensor.
Hence Theorem \ref{thm:7.3} does not provide, in general, the stable control of relative error
since for the general canonical tensors there is no uniform upper bound on the constant $C$ in the estimate
\[
\|\boldsymbol{\xi}\| \leq C \|{\bf A} \|.
\]
The problem is that it applies to the non-orthogonal canonical decomposition.
The stable decomposition can be proven in the case of partially orthogonal 
or monotone decompositions. % see \cite{VeKhorEwTuck_NLLA:15}.
\begin{corollary} \label{cor:RHOSVD_stab} (Stability of RHOSVD)
Assume the conditions of Theorem \ref{thm:7.3} are satisfied. 
(A) Suppose that one  of the matrices
$U^{(\ell)}$, say $U^{(1)}$, is orthogonal. Then the RHOSVD error can be bounded by
\begin{equation}\label{errorRHOSVD_orthog}
\|{\bf A}- {\bf A}_{({\bf r})}^0\| \leq C \|{\bf A}\| \sum\limits_{\ell=1}^d
(\sum\limits_{k=r_\ell +1}^{\min(n,R)} \sigma_{\ell,k}^2)^{1/2},
%,\quad \|\boldsymbol{\xi}\| = \sqrt{\sum\limits_{\nu =1}^{R} \xi_{\nu}^2}.
\end{equation}
with the constant $C=1$. \\
(B)  Let decomposition (\ref{eqn:can_A3}) be monotone, i.e. 
all coefficients and skeleton vectors  have non-negative values. 
Then (\ref{errorRHOSVD_orthog}) holds with the constant $C$ that does not depend on ${\bf A}$.
 \end{corollary}
\begin{proof}
(A) The partial orthogonality assumption combined with normalization constraints for 
the canonical skeleton vectors imply 
\[
    \sum\limits_{\nu =1}^{R} \xi_{\nu}^2 =  \|\boldsymbol{\xi}\|^2 = \|{\bf A} \|^2,
\]
then the result follows by (\ref{error bound CtoT}). \\
(B) In case of {\it monotone} decomposition we conclude 
that the pairwise scalar product of all summands in (\ref{eqn:can_A3}) is non-negative, 
while the norm of each $\nu$-term  
is equal to $\xi_{\nu}$. Then the result follows by the upper bound
\[
 (u_1,u_1) +\cdots + (u_R,u_R) \leq (\sum\limits_{\nu =1}^{R} u_\nu,\sum u_\nu),
\]
which holds for vectors with non-negative entries applied in the case of $R$ summands. 
\end{proof}

Clearly, the orthogonality assumption may lead to slightly higher separation rank, 
however, this constructive decomposition stabilizes the RHOSVD approximation method applied to 
 the canonical format tensor (i.e., it allows the stable control of relative error).
The case of monotone canonical sums typically arises in the sinc-based canonical approximation 
to radially symmetric Green's kernels by a sum of Gaussians. 
On the other hand, in long term computational practice the numerical instability of RHOSVD 
approximation was not observed in case of physically relevant data.

\subsection{Mixed Tucker tensor format and Tucker-to-CP transform} \label{ssec:MixedTuckCp}
 
In the procedure for the canonical tensor rank reduction the goal is to
have a result in a canonical tensor format with a smaller rank.
By converting the core tensor to CP format, one can use 
the mixed two-level Tucker data format \cite{KhKh:06,VeKh_Diss:10},
or canonical CP format.
 Figure \ref{figs1:Tucker_Can} illustrates the computational scheme of the two-level Tucker approximation. 
 \begin{figure}[h]
 \centering
\includegraphics[width=10cm]{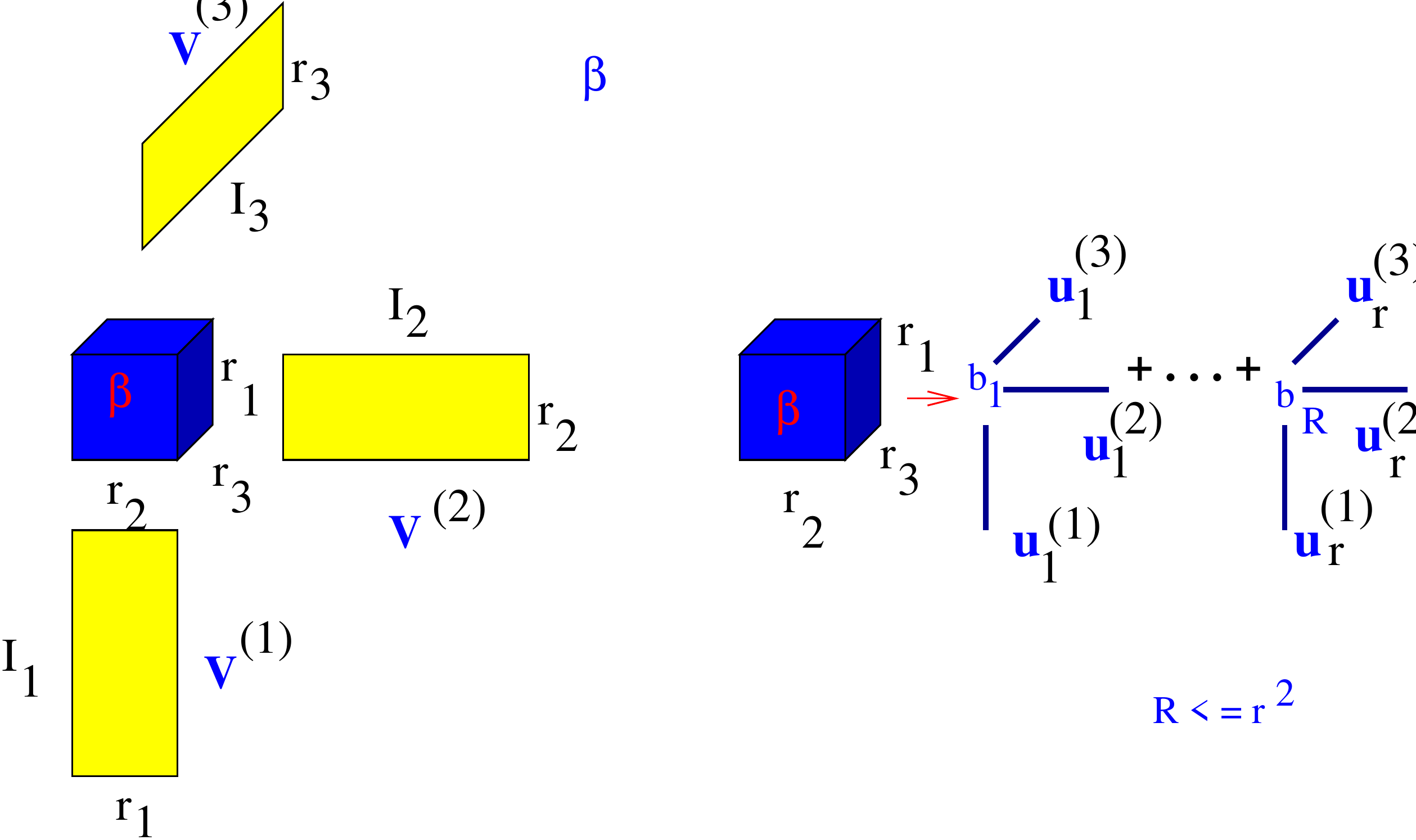} 
\caption{Mixed Tucker-to-canonical decomposition.}
\label{figs1:Tucker_Can}
\end{figure}

Next lemma describes the approximation of the  Tucker tensor by using 
canonical representation \cite{KhKh:06,VeKh_Diss:10}.
\begin{lemma}\label{lem:7.4} (Mixed Tucker-to-canonical  approximation, \cite{VeKh_Diss:10}).\\
(A) Let the target tensor ${\bf A}$  have the form
$
{\bf A}={\boldsymbol{\beta}}\times_{1}V^{(1)}\times_{2}
...   \times_{d} V^{(d)} \in \mbox{\boldmath{$\mathcal{T}$}}_{{\bf r},{\bf n}},
$
 with the orthogonal side-matrices
 $V^{(\ell)}=[{\bf v}^{(\ell)}_1 \ldots {\bf v}^{(\ell)}_{r_\ell} ]
\in \mathbb{R}^{ n \times  r_\ell }$  and  
${\boldsymbol{\beta}} \in \mathbb{R}^{r_1\times...\times r_d}$.
Then, for a given  $R \leq \min\limits_{1\leq \ell \leq d} \overline{r}_\ell$, 
%(see (\ref{can_tensor_rank})), 
\begin{equation} \label{two-level min}
  \min\limits_{{\bf Z} \in \mbox{\boldmath{$\mathcal{C}$}}_{R,{\bf n}}} \|{\bf A}- {\bf Z} \|
  = \min\limits_{{\boldsymbol{\mu}}\in 
\mbox{\boldmath{$\mathcal{C}$}}_{R,{\bf r}}} 
\|{\boldsymbol{\beta}}-{\boldsymbol{\mu}} \|,
\end{equation}
where $\overline{n}_\ell$ is the single-hole product of dimension-modes,
\begin{equation}
\label{eq:single_hole}
\overline{n}_\ell = 
n_1 \cdots  n_{\ell-1} n_{\ell+1}\cdots n_d
\end{equation}.
(B) Assume that there exists the best rank-$R$ approximation
$ {\bf A}_{(R)}\in \mbox{\boldmath{$\mathcal{C}$}}_{R,{\bf n}}$ of 
${\bf A}$, then there is the best rank-$R$  approximation 
$ {\boldsymbol{\beta}}_{(R)}\in \mbox{\boldmath{$\mathcal{C}$}}_{R,{\bf r}}$
of ${\boldsymbol{\beta}}$, such that
\begin{equation}\label{two-level relation}
 {\bf A}_{({R})}= {\boldsymbol{\beta}}_{(R)}\times_{1}V^{(1)}\times_{2} ... 
  \times_{d} V^{(d)}.
\end{equation}
\end{lemma}
\begin{proof} (A)
The canonical vectors $ {\bf y}_k^{(\ell)}$ of any test 
element on the left-hand side of (\ref{two-level min}),
\begin{equation} \label{test elem}
{\bf Z} =  \sum\limits_{k=1}^{R}\lambda_k \; {\bf y}_{k}^{(1)}\otimes
... \otimes {\bf y}_{k}^{(d)} \in \mbox{\boldmath{$\mathcal{C}$}}_{R,{\bf n}},
\end{equation}
 can be chosen in $span\{ {\bf v}^{(\ell)}_1, \ldots, {\bf v}^{(\ell)}_{r_\ell}\}$, that means
\begin{equation} \label{span V}
{\bf y}^{(\ell)}_{k}=\sum\limits_{m=1}^{r_\ell}\mu_{k,m}^{(\ell)} {\bf v}^{(\ell)}_m,
\quad k=1, \ldots ,R, \; \ell=1,...,d.
\end{equation}
Indeed, assuming 
\[
{\bf y}^{(\ell)}_{k}=\sum\limits_{m=1}^{r_\ell}\mu_{k,m}^{(\ell)} {\bf v}^{(\ell)}_m + 
{\bf e}^{(\ell)}_{k} \quad \mbox{with} \quad 
{\bf e}^{(\ell)}_{k}\bot span\{ {\bf v}^{(\ell)}_1, \ldots, {\bf v}^{(\ell)}_{r_\ell}\},
\]
we conclude that ${\bf e}^{(\ell)}_{k}$ does not effect the cost function
in (\ref{two-level min}) because of the orthogonality of $ V^{(\ell)}$.
Hence, setting ${\bf e}^{(\ell)}_{k}=0$, and plugging (\ref{span V})
in (\ref{test elem}), we arrive at the desired Tucker decomposition of ${\bf Z}$,
$
 {\bf Z}= \boldsymbol{\beta}_z \times_1 V^{(1)} \times_2 \ldots \times_d V^{(d)},
\quad  \boldsymbol{\beta}_z \in \mbox{\boldmath{$\mathcal{C}$}}_{R,{\bf r}}.
$
This implies 
\begin{eqnarray*} 
\| {\bf A} - {\bf Z} \|^2 =  \| (\boldsymbol{\beta}_z -\boldsymbol{\beta} ) \times_1 V^{(1)} \times_2 
\ldots \times_d V^{(d)} \|^2 
 =  \| \boldsymbol{\beta} - \boldsymbol{\beta}_z \|^2 \geq 
\min\limits_{\boldsymbol{\mu} \in \mbox{\boldmath{$\mathcal{C}$}}_{R,{\bf r}}} 
\| \boldsymbol{\beta} -\boldsymbol{\mu} \|^2 .
\end{eqnarray*}
On the other hand, we have
\begin{eqnarray*} 
 \min\limits_{{\bf Z} \in \mbox{\boldmath{$\mathcal{C}$}}_{R,{\bf n}} } 
\| {\bf A} - {\bf Z} \|^2 \leq 
\min\limits_{ \boldsymbol{\beta}_z\in \mbox{\boldmath{$\mathcal{C}$}}_{R,{\bf r}} }
\| (\boldsymbol{\beta} - \boldsymbol{\beta}_z)  
\times_1 V^{(1)} \times_2 \ldots \times_d V^{(d)} \|^2
=  \min\limits_{ \boldsymbol{\mu} \in \mbox{\boldmath{$\mathcal{C}$}}_{R,{\bf r}} }
\| \boldsymbol{\beta} - \boldsymbol{\mu} \|^2.
\end{eqnarray*}
This proves (\ref{two-level min}). 

(B) Likewise, for any minimizer
${\bf A}_{(R)}\in  \mbox{\boldmath{$\mathcal{C}$}}_{R,{\bf n}}$ 
in the right-hand side in (\ref{two-level min}), one obtains 
\[
{\bf A}_{(R)} = \boldsymbol{\beta}_{(R)} \times_{1} V^{(1)}\times_{2} V^{(2)} ... 
  \times_{d} V^{(d)}
\]
  with the respective rank-$R$ core tensor
  $
  \boldsymbol{\beta}_{(R)} = \sum\limits_{k=1}^{R} \lambda_k
   {\bf u}_{k}^{(1)}\otimes ... \otimes {\bf u}_{k}^{(d)}\in
  \mbox{\boldmath{$\mathcal{C}$}}_{R,{\bf r}}.
  $
  Here ${\bf u}_{k}^{(\ell)}= \{\mu_{k,m_\ell}^{(\ell)} \}_{m_\ell=1}^{r_\ell}
  \in \mathbb{R}^{r_\ell}$,  are  calculated by plugging representation 
  (\ref{span V}) in (\ref{test elem}), and then by changing the order of summation,
\begin{align*}
 {\bf A}_{(R)}  &= \sum\limits_{k=1}^{R}\lambda_k  {\bf y}_{k}^{(1)}\otimes
... \otimes {\bf y}_{k}^{(d)}
=\sum\limits_{k=1}^{R}\lambda_k \left(\sum\limits_{m_1=1}^{r_1}
\mu_{k,m_1}^{(1)} {\bf v}^{(1)}_{m_1}\right)\otimes
... \otimes \left(\sum\limits_{m_d=1}^{r_d}\mu_{k,m_d}^{(d)} {\bf v}^{(d)}_{m_d}\right)  \\
&= \sum\limits_{m_1=1}^{r_1}...\sum\limits_{m_d=1}^{r_d}
\left\{\sum\limits_{k=1}^{R}\lambda_k
\prod\limits_{\ell=1}^d \mu_{k,m_\ell}^{(\ell)}\right\}
{\bf v}^{(1)}_{m_1}\otimes ...\otimes {\bf v}^{(d)}_{m_d}.
\end{align*}  
Now (\ref{two-level relation}) implies that
$$
\| {\bf A} - {\bf A}_R \| = \|\boldsymbol{\beta} - \boldsymbol{\beta}_{(R)} \|,
$$
since the $\ell$-mode multiplication with the orthogonal side matrices $V^{(\ell)}$
does not change the cost functional. 
Inspection of the left-hand side in (\ref{two-level min}) indicates that the
latter equation ensures that 
$\boldsymbol{\beta}_{(R)}$ is, in fact,  the minimizer of the right-hand side in (\ref{two-level min}).
\end{proof}

Combination of Theorem \ref{thm:7.3} and Lemma \ref{lem:7.4} opens the way to the
rank optimization of canonical tensors with the large mode-size arising, for example, 
in the grid-based numerical methods for multidimensional PDEs with non-regular (singular) solutions.
In such applications the univariate grid-size (i.e. the mode-size) 
may be about $n=10^4$ and even larger.

Notice that the Tucker (for moderate $d$) and canonical formats allow to perform basic multi-linear algebra  
using {\it one-dimensional operations},  
thus reducing the exponential scaling in $d$. Rank-truncated transforms between different formats can be applied in 
multi-linear algebra on mixed tensor representations as well, see Lemma \ref{lem:7.4}. 
The particular application to tensor convolution  in many dimensions was 
discussed, for example,  in \cite{Khor_book:18,Khor_bookQC:18}.

We summarize that the direct methods of tensor approximation can be classified by:
\begin{itemize}
 \item [(1)] {Analytic Tucker approximation} to some classes of function-related
$d$th order tensors ($d\geq 2$),  say, by multivariate polynomial interpolation \cite{Khor_book:18}.
\item [(2)] Sinc quadrature based approximation methods in the canonical format applied to a class of analytic 
function related tensors \cite{HaKhtens:04I}.
%\item [(3)] Combination of  ACA-type and truncated SVD in the matrix case ($d=2$).
\item [(3)] Truncated HOSVD and  RHOSVD, for quasi-optimal Tucker approximation of the full-format, respectively, 
canonical tensors \cite{KhKh3:08}.
\end{itemize} 
Direct analytic approximation methods by sinc quadrature/interpolation are of principal importance.
Basic examples are given by the tensor representation of Green's kernels, 
the elliptic operator inverse and analytic matrix-valued functions.
In all cases the algebraic methods for rank reduction by the ALS-type iterative Tucker/canonical approximation can be applied.

Further improvement and enhancement of algebraic tensor approximation methods can be based on the combination of
advanced nonlinear iteration, multigrid tensor methods, greedy algorithms, 
hybrid tensor representations and the use of new problem adapted tensor formats.

\subsection{Tucker-to-canonical transform} \label{ssec:Tuck2Can}

In the rank reduction scheme for the canonical rank-$R$ tensors, 
we use successively the canonical-to-Tucker (C2T) transform and then the
Tucker-to-canonical (T2C) tensor approximation. 
%Next, we give two useful remarks which characterize the canonical representation of the full format tensors \cite{KhKh3:08}.
% Denote by $\overline{n}_\ell$ the single-hole product of $\ell$-mode dimensions
% \begin{equation} \label{singleHole}
% \overline{n}_\ell = n_1 \cdots  n_{\ell-1} n_{\ell+1}\cdots n_d.
% \end{equation}
%\begin{remark}\label{single_hole}

First, we notice that the canonical rank of a tensor  ${\bf A} \in \mathbb{V}_{\bf n}$ has the upper bound, see 
\cite{VeKh_Diss:10,KhKh3:08},
\begin{equation}
\label{can_tensor_rank}
 R\leq \min\limits_{1\leq \ell \leq d} \overline{n}_\ell,
\end{equation}
where $\overline{n_\ell}$ is given by (\ref{eq:single_hole}).
%\end{remark}
 Rank bound (\ref{can_tensor_rank}) applied to the Tucker core tensor of the size $r\times r \times r$, indicates that the ultimate 
canonical rank of a large-size tensor in $\mathbb{V}_{\bf n}$ has the upper bound $r^2$.

The following remark shows that the maximal canonical rank of the Tucker core of 3rd order
tensor can be easily reduced to the value less than $r^2$ by the SVD-based procedure applied to the matrix slices of the Tucker core 
tensor $\boldsymbol{\beta}$.
Though, being not practically attractive for arbitrary high order tensors, 
the simple algorithm described in Remark \ref{rem_can_red} below is proved to be
 useful for the treatment of small size 3rd order Tucker core tensors
within the rank reduction algorithms described in the previous sections. 
\begin{remark}\label{rem_can_red} \cite{VeKh_Diss:10,KhKh3:08}
Let $d=3$ for the sake of clearness.
There is a simple procedure based on SVD to reduce the
canonical  rank of the core tensor $\boldsymbol{\beta}$, within the accuracy $\varepsilon >0 $.  
Denote by $B_m\in \mathbb{R}^{r\times r}$, $m=1,\ldots ,r$
the two-dimensional slices of $\boldsymbol{\beta}$ in each fixed mode and represent
\begin{equation}  \label{T2C_slice}
\boldsymbol{\beta} =\sum\limits_{m=1}^r B_m \otimes {\bf z}_m, \quad {\bf z}_m \in \mathbb{R}^r,
\end{equation}
where ${\bf z}_m(m)=1$, ${\bf z}_m(j)=0$ for $j=1,\ldots ,r$, $j\neq m$ (there are exactly $d$ possible
decompositions). Let $p_m$ be the minimal integer, such that the singular values of
$B_m$ satisfy $\sigma^{(m)}_k \leq \frac{\varepsilon}{r^{3/2}} $ for $k = p_{m}+1,...,r$
(if $\sigma^{(m)}_r > \frac{\varepsilon}{r^{3/2}}$, then set $p_m=r$). Then, denoting by
\begin{equation*}
 \label{T2C_SVD}
B_{p_m}= \sum\limits_{k_m=1}^{p_m}\sigma^{(m)}_{k_m} {\bf u}_{k_m}\otimes {\bf v}_{k_m},
\end{equation*}
the corresponding rank-$p_m$ approximation to $B_m$ (by truncation
of $\sigma^{(m)}_{p_m+1},\ldots ,\sigma^{(m)}_{r}$), we arrive at the
rank-$R$ canonical approximation to $\boldsymbol{\beta}$,
\begin{equation} \label{beta_R}
\boldsymbol{\beta}_{(R)} := \sum\limits_{m=1}^r B_{p_m} \otimes {\bf z}_m, \quad
{\bf z}_m \in \mathbb{R}^r,
\end{equation}
providing the error estimate
\[
\|\boldsymbol{\beta}- \boldsymbol{\beta}_{(R)} \| \leq \sum\limits^r_{m=1}
\| B_m - B_{p_m} \| = \sum\limits^r_{m=1} \sqrt{\sum\limits^r_{k_m =p_{m}+1} (\sigma^{(m)}_{k_m}})^2
\leq \sum\limits^r_{m=1}\sqrt{r \frac{\varepsilon^2}{r^3} } = \varepsilon
\]
Representation (\ref{beta_R}) is a sum of rank-$p_m$ terms so that the total rank is bounded by
$R \leq p_1+...+p_r \leq r^2$. 
The approach can be extended to arbitrary $d\geq 3 $ with the bound $R \leq r^{d-1}$.
\end{remark}
\begin{figure}[tbh]
  \begin{center}
  \includegraphics[width=10cm]{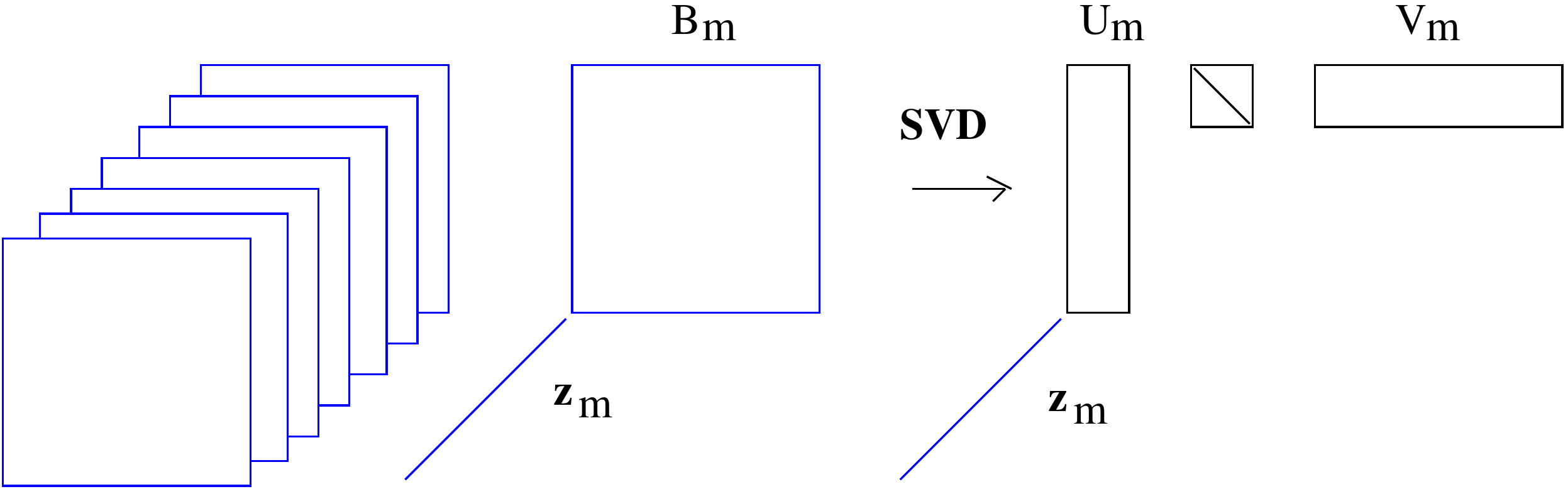}
\end{center}
\caption{ Tucker-to-canonical decomposition for a small core tensor, see Remark \ref{rem_can_red}.}    
\label{fig:C2T_slices} 
\end{figure}
 
Figure \ref{fig:C2T_slices} illustrates the canonical decomposition of the core tensor by
using the SVD of slices $B_m$ of the core tensor $\boldsymbol{\beta}$, yielding   
 matrices $U_m =\{{\bf u}_{k_m}\}_{k=1}^{p_m}$, $V_m =\{{\bf v}_{k_m}\}_{k=1}^{p_m}$ and 
 a diagonal  matrix of small size $ p_m\times p_m$ containing the truncated singular values. 
 It also shows the vector ${\bf z}_m =[0,\ldots, 0, 1, 0, \ldots, 0]$, 
 containing all entries equal to $0$ except $1$ at the m-th position.

 \section{Calculation of 3D  integrals with the Newton kernel}\label{ssec:Multi_Int}

The first application of the RHOSVD was calculation of the 3D grid-based Hartree potential operator 
in the Hartree-Fock equation,
\begin{equation}\label{eq:HP}
{V_{H}({ x}):= \int_{\mathbb{R}^3}
\frac{\rho({ y})}{\|{ x}-{ y}\|}\, d{ y}},
\end{equation} 
where the electron density, 
\begin{equation}\label{eq:rho}
 \rho(x) = 2 \sum\limits_{a=1}^{N_{orb}}(\varphi_a )^2
\end{equation},  
is represented in terms of molecular orbitals,
presented in the Gaussian-type  basis (GTO),
\[
\varphi_a(x)=\sum\limits_{k =1}^{N_b} c_{a, k } g_k(x)
\] 
The Hartree potential describes the repulsion energy of the electrons in a molecule.
The intermediate goal here is the calculation of the so-called Coulomb matrix, 
\[
J_{km}:=\int_{\mathbb{R}^3} g_k({x}) g_m({x}) V_H({x}) d {x}, \quad
k, m =1,\ldots N_b\quad x \in \mathbb{R}^3,
\]
which represents the Hartree potential in the given GTO basis.
 
In fact, calculation of this 3D convolution 
operator  with the Newton kernel, requires high accuracy and it should be repeated  multiply in the course 
of the iterative solution of the Hartree-Fock nonlinear eigenvalue problem. 
The presence of nuclear cusps in the electron density 
makes additional challenge to computation of the Hartree potential operator. Traditionally, these calculations  
are based on involved analytical evaluation of the corresponding integral in a separable Gaussian basis set 
by using erf function.  
Tensor-structured calculation of the multidimensional convolution integral operators with the Newton kernel 
have been introduced in \cite{KhKh3:08,KKF:09,VeKh_Diss:10}.
% 
% ,
% where on the examples of the Hartree and exchange operators  in the Hartree-Fock equation,
% it was  shown that calculation of the  3D and 6D convolution integrals can be reduced to a combination of 
% 1D Hadamard products, 1D convolutions and 1D scalar products.  
% However, the main challenge in these calculations was the reduction of the canonical tensors increasing in the course of tensor
% calculations. For solving this problem the RHOSVD and the C2T plus T2C transforms were introduced in \cite{KhKh3:08}.  

The molecule is embedded in a computational box $ \Omega=[-b,b]^3 \in \mathbb{R}^3$. % as in Fig. \ref{fig:3Dbox}. 
%For a given discretization parameter $n \in \mathbb{N}$,  we use the 
The equidistant $n\times n \times n$ tensor grid
$\omega_{{\bf 3},n}=\{x_{\bf i}\} $, ${\bf i} \in {\cal I} :=\{1,...,n\}^3 $,
with the mesh-size $h=2b/(n+1)$ is used. 
% Calculations are performed in the same Gaussian basis, which is used for %analytical calculations of the same integrals.
In calculations of integral terms, 
 the Gaussian basis functions $g_k (x),\; x\in\mathbb{R}^3$,
are approximated by sampling their values at the centers of discretization intervals  using one-dimensional 
piecewise constant basis functions
$g_k (x) \approx  \overline{g}_k (x)=\prod^3_{\ell=1} \overline{g}_k^{(\ell)} (x_{\ell})$, $\ell=1,  2,  3$, 
yielding their rank-$1$ tensor representation, \cn
\begin{equation}
\label{disc_GTO}
 {\bf G}_k= {\bf g}_k^{(1)} \otimes {\bf g}_k^{(2)} \otimes {\bf g}_k^{(3)}\in \mathbb{R}^{n\times n \times n},
 \quad k=1,...,N_b.
\end{equation}
 Given the discrete tensor representation of basis functions (\ref{disc_GTO}), the electron density is 
approximated using 1D Hadamard products of rank-$1$ tensors as
\begin{equation}
 \label{eq:rho_ten}
\rho \approx \Theta=2 \sum\limits^{N_{orb}}_{a=1} \sum\limits^{N_b}_{k=1}
\sum\limits^{N_b}_{m=1}
c_{a, m}c_{a, k}({\bf g}_{  k}^{(1)} \odot {\bf g}_{  m}^{(1)}) \otimes 
\cdots
%({\bf g}_{ k}^{(2)} \odot {\bf g}_{  m}^{(2)}) 
%
\otimes ({\bf g}_{  k}^{(3)} \odot {\bf g}_{  m}^{(3)})
\in \mathbb{R}^{n\times n \times n}.
\end{equation}
For convolution   operator, the representation   of the Newton kernel $\frac{1}{\|{ x}-{ y}\|}$ by a canonical rank-$R_N$
tensor \cite{Khor_book:18} is used (see section \ref{ssec:Theory_RS_sinc}  for details),
\begin{equation} \label{eqn:NewtCan}
{\bf P}_R=
\sum\limits_{q=1}^{R_N} {\bf p}^{(1)}_q \otimes {\bf p}^{(2)}_q \otimes {\bf p}^{(3)}_q
\in \mathbb{R}^{n\times n \times n}.
\end{equation}
 The initial rank of the electron density in the canonical tensor format $ \Theta$ in (\ref{eq:rho}) 
is large even for small molecules.
Rank reduction by using RHOSVD C2T plus T2C reduces the rank
$ \Theta \mapsto \Theta'$ by several orders of magnitude, from $N_b^2/2$ to $R_{\rho} \ll N_b^2/2$, 
from $\sim 10^4$ to $\sim 10^2$ .
Then the 3D tensor representation of the Hartree potential is calculated by using
the 3D tensor product convolution, which is a sum of tensor products of 1D convolutions, 
\[
 V_H \approx {\bf V}_H = \Theta' \ast {\bf P}_R =
\sum\limits_{ m= 1}^{R_{\rho}}
\sum\limits_{q=1}^{R_N} c_m \left(  {\bf u}^{(1)}_m \ast {\bf p}^{(1)}_q \right) \otimes
\left(  {\bf u}^{(2)}_m \ast {\bf p}^{(2)}_q \right) 
\otimes
\left(  {\bf u}^{(3)}_m \ast {\bf p}^{(3)}_q \right).
\]
The Coulomb matrix entries $J_{k m} $ are obtained by 1D scalar products
of ${\bf V}_H$ with the Galerkin basis consisting of rank-$1$ tensors,
\[
J_{km}% :=\int_{\mathbb{R}^3} g_k({x}) g_m({x}) V_H({x}) d {x}
% k, m =1,\ldots N_b\quad x \in \mathbb{R}^3,
\approx \langle {\bf G}_k \odot {\bf G}_m, {\bf V}_H \rangle,
\quad k, m =1,\ldots N_b.
\]
The cost of 3D tensor product convolution is $O( n\log n)$ instead of $O(n^3 \log n)$ for the
standard benchmark 3D convolution using the 3D FFT.  Table \ref{tab:TConvsFFT} shows CPU times (sec)  for
the Matlab computation of $V_H$ for H$_2$O molecule \cite{KhKh3:08} on a SUN station using 8 Opteron
Dual-Core/2600 processors (times for 3D FFT  for $n\geq 1024$ are obtained by extrapolation).
C2T shows the time for the canonical-to-Tucker rank reduction.
 
 The grid-based tensor calculation of the multidimensional integrals in
quantum chemistry provides the required high accuracy by using large grids and
the ranks are controlled by the required $\varepsilon $ in the rank truncation
algorithms. 
The results of the tensor-based calculations have been compared with the results
of the benchmark standard computations by the MOLPRO package. It was shown
that the accuracy is of the order of $10^{-7}$ hartree in the resulting ground state energy, 
see \cite{VeKh_Diss:10,Khor_bookQC:18}. 
\begin{table}[tbh]
\begin{center}%
\begin{tabular}
[c]{|r|r|r|r|r|r|}%
\hline
$n^3$      & $1024^3$ & $2048^3$ & $4096^3$ & $8192^3$ & $16384^3$\\
 \hline
FFT$_3$      &  $\sim 6000$ & --&--& -- & $\sim 2$ years\\
 \hline
$C\ast C$     & $8.8$ & $20.0$ & $61.0 $ & $157.5 $ & $ 299.2 $\\
\hline
{\footnotesize C2T }     & $6.9$ & $10.9$ & $20.0$ & $37.9$ & $86.0$\\
\hline
 \end{tabular}
\caption{Times (sec) for the C2T transform and the 3D tensor product convolution vs. 3D FFT convolution.}
\end{center}
\label{tab:TConvsFFT}
\end{table}

 \section{RHOSVD in the range-separated tensor formats}\label{sec:RHOSVD_RSf}
 
 The range-separated (RS) tensor formats have been introduced in \cite{BKK_RS:17} as the constructive tool
 for low-rank tensor representation (approximation) of function related data discretized on Cartesian grids in $\mathbb{R}^d$, 
 which may have multiple  singularities or cusps. 
 Such highly non-regular data typically arise in computational quantum chemistry, in many-particle  
 dynamics simulations and many-particle electrostatics calculations, in protein modeling and in data science.
 The key idea of the RS representation is the splitting of the short- and long-range parts in the functional data and 
 further low-rank approximation of the rather regular long-range part in the classical tensor formats.
 
 In this concern RHOSVD method becomes  an essential ingredient of the rank reduction algorithms for the  
 ``long-range'' input tensor, which usually inherits the large initial rank.

 \subsection{Low-rank approximation of radial functions} \label{ssec:Theory_RS_sinc}
 
 First, we recall the grid-based method for the low-rank canonical  
representation of a spherically symmetric kernel functions $p(\|x\|)$, 
$x\in \mathbb{R}^d$ for $d=2,3,\ldots$, by its projection onto the finite set
of basis functions defined on tensor grid. 
The approximation theory by a sum of Gaussians for the class of analytic potentials 
$p(\|x\|)$  was presented in 
\cite{Stenger:93,HaKhtens:04I,Khor1:06,Khor_book:18}.
The particular numerical schemes for rank-structured representation of the Newton 
and Slater  kernels 
\begin{equation}\label{eqn:NewtYukaw}
p(\|x\|)=\frac{1}{4 \pi \|x\|},\quad  \mbox{  and  }\quad
%p(\|x\|)=\frac{e^{-\lambda \|x\|}}{4 \pi \|x\|},\quad
p(\|x\|)=e^{-\lambda \|x\|},\quad
 x\in \mathbb{R}^3,
\end{equation}
discretized on a fine 3D  Cartesian grid
in the form of low-rank canonical tensor was described in 
\cite{HaKhtens:04I,Khor1:06}.

In what follows, for the ease of exposition, we confine ourselves to the case $d=3$.
%though that the sinc-quadrature based separable approximation 
%of the radial functions apply for the arbitrary dimension $d$.
In the computational domain  $\Omega=[-b,b]^3$, 
let us introduce the uniform $n \times n \times n$ rectangular Cartesian grid $\Omega_{n}$
with mesh size $h=2b/n$ ($n$ even).
Let $\{ \psi_\textbf{i} =\prod_{\ell=1}^3 \psi_{i_\ell}^{(\ell)}(x_\ell) \}$ 
be a set of tensor-product piecewise constant basis functions, 
% \begin{equation*}\label{eqn:basis}
% \psi_\textbf{i}(\textbf{x})=\prod_{\ell=1}^3 \psi_{i_\ell}^{(\ell)}(x_\ell),
% \end{equation*}
labeled by  the $3$-tuple index ${\bf i}=(i_1,i_2,i_3)$, 
$i_\ell \in I_\ell=\{1,...,n\}$, $\ell=1,\, 2,\, 3 $.
The generating kernel $p(\|x\|)$ is discretized by its projection onto the basis 
set $\{ \psi_\textbf{i}\}$
in the form of a third order tensor of size $n\times n \times n$, defined entry-wise as
\begin{equation}  \label{eqn:galten}
% {\bf P} := [p_{i_1  i_2  i_3}] \in \mathbb{R}^{n \times n \times n },
% \quad
\mathbf{P}:=[p_{\bf i}] \in \mathbb{R}^{n\times n \times n},  \quad
%\mathbf{P}:=\pc{p_\tb{i}}_{\tb{i} \in \mathcal{I}}\in \mathbb{R}^{n\times n \times n},  \quad
 p_{\bf i} = 
%p_{i_1  i_2  i_3}=
\int_{\mathbb{R}^3} \psi_{\bf i} ({x}) p(\|{x}\|) \,\, \mathrm{d}{x}.
% \frac{\psi_{i_1 i_2 i_3}(\textbf{x})}{\|\mathbf{x}\|} \,\, \mathrm{d}\tb{x}.
 % \quad \mbox{where}\quad \Omega_\tb{i}=\operatorname{supp}( \psi_\tb{i}).
\end{equation}

The low-rank canonical decomposition of the $3$rd order tensor $\mathbf{P}$ is based 
on using exponentially fast convergent 
$\operatorname*{sinc}$-quadratures for approximating the Laplace-Gauss transform 
to the analytic function $p(z)$, $z \in \mathbb{C}$, specified by a certain weight $\widehat{p}(t) >0$,
\begin{align} \label{eqn:laplace} 
p(z)=\int_{\mathbb{R}_+} \widehat{p}(t) e^{- t^2 z^2} \,\mathrm{d}t \approx
\sum_{k=-M}^{M} p_k e^{- t_k^2 z^2} \quad \mbox{for} \quad |z| > 0,\quad z \in \mathbb{R},
\end{align} 
with the proper choice of the quadrature points $t_k$ and weights $p_k$.
% Under the assumption $0< a \leq |z |  < \infty$
% this quadrature can be proven to provide an exponential convergence rate in $M$
% for a class of analytic functions $p(z)$. 
 The $sinc$-quadrature based approximation to generating function by 
using the short-term Gaussian sums in (\ref{eqn:laplace})  %, (\ref{eqn:hM}) 
are applicable to the class of analytic functions
in certain strip $|z|\leq D $ in the complex plane, such that on the real axis these functions decay
polynomially or exponentially. We refer to basic results in 
\cite{Stenger:93,Braess:BookApTh,HaKhtens:04I}, 
where the exponential convergence of the $sinc$-approximation in the number of terms 
(i.e., the canonical rank) was analyzed for certain classes of analytic integrands.

Now, for any fixed $x=(x_1,x_2,x_3)\in \mathbb{R}^3$, 
such that $\|{x}\| > a > 0$, %$\|\mathbf{x}\|$
we apply the $\operatorname*{sinc}$-quadrature approximation (\ref{eqn:laplace})  % , (\ref{eqn:hM}) 
to obtain the separable expansion
\begin{equation} \label{eqn:sinc_Newt}
 p({\|{x}\|}) =   \int_{\mathbb{R}_+} \widehat{p}(t)
e^{- t^2\|{x}\|^2} \,\mathrm{d}t  \approx 
\sum_{k=-M}^{M} p_k e^{- t_k^2\|{x}\|^2}= 
\sum_{k=-M}^{M} p_k  \prod_{\ell=1}^3 e^{-t_k^2 x_\ell^2},
\end{equation}
providing an exponential convergence rate in $M$,
% Under the assumption $0< a \leq \|{x}\| \leq A < \infty$
% this approximation provides the exponential convergence rate in $M$,
\begin{equation} \label{eqn:sinc_conv}
\left|p({\|{x}\|}) - \sum_{k=-M}^{M} p_k e^{- t_k^2\|{x}\|^2} \right|  
\le \frac{C}{a}\, \displaystyle{e}^{-\beta \sqrt{M}},  
\quad \text{with some} \ C,\beta >0.
\end{equation}

In the case of Newton kernel, we have $p(z)=1/z$, $\widehat{p}(t)=\frac{2}{\sqrt{\pi}}$, so that
the Laplace-Gauss transform representation reads
\begin{equation} \label{eqn:Laplace_Newton}
 \frac{1}{z}= \frac{2}{\sqrt{\pi}}\int_{\mathbb{R}_+} e^{- z^2 t^2 } dt, \quad 
 \mbox{where}\quad z=\|x\|, \quad x \in \mathbb{R}^3,
 %z=\sqrt{x_1^2 + x_2^2  + x_3^2}.
\end{equation}
which can be approximated by the sinc quadrature (\ref{eqn:sinc_Newt}) with the particular choice 
of quadrature points $t_k$, providing the exponential convergence rate as 
in (\ref{eqn:sinc_conv}), \cite{HaKhtens:04I,Khor1:06}.

In the case of Yukawa potential the Laplace Gauss transform reads   %representation (\ref{eqn:Laplace_Newton})
 \begin{equation} \label{eqn:Laplace_Yukawa}
  \frac{e^{-\kappa z}}{z}= \frac{2}{\sqrt{\pi}}\int_{\mathbb{R}_+} e^{-\kappa^2/t^2} e^{- z^2 t^2 } dt, \quad 
 \mbox{where}\quad z=\|x\|, \quad x \in \mathbb{R}^3.
 %z=\sqrt{x_1^2 + x_2^2  + x_3^2}.
 \end{equation}
 The analysis of the sinc quadrature approximation error for this case can 
 be found, in particular, in \cite{Khor1:06,Khor_book:18}, \S2.4.7.

Combining \eqref{eqn:galten} and \eqref{eqn:sinc_Newt}, and taking into account the 
separability of the Gaussian basis functions, we arrive at the low-rank 
approximation to each entry of the tensor $\mathbf{P}=[p_{\bf i}]$,
\begin{equation*} \label{eqn:C_nD_0}
 p_{\bf i} \approx \sum_{k=-M}^{M} p_k   \int_{\mathbb{R}^3}
 \psi_{\bf i}({x}) e^{- t_k^2\|{x}\|^2} \mathrm{d}{x}
=  \sum_{k=-M}^{M} p_k  \prod_{\ell=1}^{3}  \int_{\mathbb{R}}
\psi^{(\ell)}_{i_\ell}(x_\ell) e^{- t_k^2 x^2_\ell } \mathrm{d} x_\ell.
\end{equation*}
Define the vector (recall that $p_k >0$) 
\begin{equation} \label{eqn:galten_int}
\textbf{p}^{(\ell)}_k
= p_k^{1/3} \left[b^{(\ell)}_{i_\ell}(t_k)\right]_{i_\ell=1}^{n_\ell} \in \mathbb{R}^{n_\ell}
\quad \text{with } \quad b^{(\ell)}_{i_\ell}(t_k)= 
\int_{\mathbb{R}} \psi^{(\ell)}_{i_\ell}(x_\ell) e^{- t_k^2 x^2_\ell } \mathrm{d}x_\ell,
\end{equation}
then the $3$rd order tensor $\mathbf{P}$ can be approximated by 
the $R$-term ($R=2M+1$) canonical representation
\begin{equation} \label{eqn:sinc_general}
    \mathbf{P} \approx  \mathbf{P}_R =
\sum_{k=-M}^{M} p_k \bigotimes_{\ell=1}^{3}  {\bf b}^{(\ell)}(t_k)
= \sum\limits_{k=-M}^{M} {\bf p}^{(1)}_k \otimes {\bf p}^{(2)}_k \otimes {\bf p}^{(3)}_k
\in \mathbb{R}^{n\times n \times n}, \quad {\bf p}^{(\ell)}_k \in \mathbb{R}^n.
% \quad {\bf p}^{(\ell)}_k \in \mathbb{R}^n.
%\quad a_k,\, t_k \in \mathbb{R},
\end{equation}
Given a threshold $\varepsilon >0 $, in view of (\ref{eqn:sinc_conv}), we can chose  
$M=O(\log^2\varepsilon )$ such that in the max-norm
\begin{equation*} \label{eqn:error_control}
\| \mathbf{P} - \mathbf{P}_R \|  \le \varepsilon \| \mathbf{P}\|.
\end{equation*}

In the case of continuous radial function $p({\|{x}\|})$, say the Slater potential, 
we use the collocation type discretization at 
the grid points including the origin, $x=0$,  so that the univariate mode size becomes $n \to n_1=n+1$. 
In what follows, we use the same notation $\mathbf{P}_R$ in the case of collocation type tensors 
(for example, the Slater potential) so that
the particular meaning becomes clear from the context.

\subsection{The RS tensor format revisited} \label{ssec:RS_form}

 The range separated (RS) tensor format was introduced in \cite{BKK_RS:17} for efficient representation of the 
collective free-space electrostatic potential of large  biomolecules. This rank-structured tensor
representation of the collective electrostatic potential of many-particle systems of general
type allows to reduce essentially computation of their interaction energy, 
and it provides convenient form for performing other algebraic transforms.
The RS format proved to be useful for range-separated tensor representation 
of the Dirac delta \cite{BKhor_Dirac:18} in $\mathbb{R}^d$ and based on that, for regularization 
 of the Poisson-Boltzmann equation (PBE)
 by decomposition of the solution into short- and long-range parts, 
where the short-range part of the solution is evaluated by simple tensor operations without solving the PDE.
The smooth long-range part is calculated by solving the PBE with the  modified right-hand side 
by using the RS decomposition of the Dirac delta, so that now it does not contain singularities.
We refer to papers \cite{BeKhKhKwSt:18,KKKSB:21} describing the approach in details.

First, we recall the definition of the range separated (RS) tensor format, see \cite{BKK_RS:17}, 
for representation of $d$-tensors ${\bf A}  \in \mathbb{R}^{n_1\times \cdots \times n_d}$. 
 The RS format is served for the hybrid tensor approximation 
 of discretized functions with multiple cusps or singularities. 
 This allows the splitting of the target tensor onto the highly localized components 
 approximating the singularity and the component with global support that allows the low-rank tensor approximation.
 Such functions typically arise in computational quantum chemistry, in many-particle modeling 
 and in the interpolation of
 multidimensional data measured at certain set of spacial points in $\mathbb{R}^{n \times n \times n}$.
  
  \begin{definition}\label{Def:RS-Can_format} (RS-canonical tensors, \cite{BKK_RS:17}).
 The RS-canonical tensor format defines the class of $d$-tensors 
 ${\bf A}  \in \mathbb{R}^{n_1\times \cdots \times n_d}$,
represented as a sum of a rank-${R_l}$ CP tensor %${\bf U}_{long}$,
$
 {\bf U}_{long}= {\sum}_{k =1}^{R_l} \xi_k {\bf u}_k^{(1)}  \otimes \cdots \otimes {\bf u}_k^{(d)},
$
and a cumulated CP tensor ${\bf U}_{short}= {\sum}_{\nu =1}^{N_0} c_\nu {\bf U}_\nu$, such that
\begin{equation}\label{eqn:RS_Can}
 {\bf A} = {\bf U}_{long} + {\bf U}_{short}= {\sum}_{k =1}^{R_l} \xi_k {\bf u}_k^{(1)}  
 \otimes \cdots \otimes {\bf u}_k^{(d)} +
 {\sum}_{\nu =1}^{N_0} c_\nu {\bf U}_\nu, 
% \quad \mbox{with} \quad  \mbox{diam}(\mbox{supp}{\bf U}_\nu)\leq 2 \gamma,
\end{equation}
where ${\bf U}_{short}$ is
generated by the localized reference CP tensor ${\bf U}_0$, i.e., ${\bf U}_\nu = \mbox{Replica} ({\bf U}_0)$,
with $\mbox{rank} ({\bf U}_\nu)= \mbox{rank}({\bf U}_0) \leq R_s$, 
%${\bf U}_\nu = \mbox{Replica} ({\bf U}_0)$, %as in Definition \ref{def:CCT_Uniform},
where, given the threshold $\delta >0$,  the effective support of ${\bf U}_\nu$ is bounded by 
$\mbox{diam}(\mbox{supp}{\bf U}_\nu)\leq 2 n_\delta$ in the index size, 
where $n_\delta= O(1)$ is a small integer.
\end{definition}
Each RS-canonical tensor is, therefore, uniquely defined by the following parametrization:
rank-$R_l$ canonical tensor ${\bf U}_{long}$, the rank-$R_s$ reference canonical tensor ${\bf U}_0$
with the small mode size bounded by $2n_\delta$, list ${\cal J}$ of the coordinates and weights of $N_0$ 
particles in $\mathbb{R}^d$. The storage size is linear in both the dimension and the univariate grid size, 
$$
\mbox{stor}({\bf A})\leq d R_l n + (d+1)N_0 + d R_s n_\delta.
$$
The main benefit of the RS-canonical  tensor decomposition is the almost uniform bound on the CP/Tucker rank 
of the long-range part ${\bf U}_{long}= {\sum}_{k =1}^{R_l} \xi_k {\bf u}_k^{(1)}  \otimes \cdots \otimes {\bf u}_k^{(d)}$, 
in the multi-particle potential discretized on fine $n \times n \times n$ spatial grid. It was proven in \cite{BKK_RS:17}
that the canonical rank $R$ scales logarithmically in both the number of particles $N_0$ and the approximation precision,
see also Lemma \ref{prop:Slater_LR_rank}. 

Given the rank-$R$ CP decomposition (\ref{eqn:sinc_general}) based on the 
sinc-quadrature approximation (\ref{eqn:sinc_Newt}) of the discretized radial function 
$p(\|{ x} \|)$, we define the two subsets of indices, ${\cal K}_l:=\{k: t_k\leq 1\}$ and 
${\cal K}_s:=\{k: t_k > 1\}$, and then
introduce the RS-representation of this tensor as follows,
\begin{equation}\label{eqn:RS_repres}
  {\bf P}_R= {\bf P}_{R_l} + {\bf P}_{R_s}, \quad R=R_l+R_s, \quad R_l = \# {\cal K}_l, \quad R_s = \# {\cal K}_s,
\end{equation}
where
\[
  {\bf P}_{R_l}:=   \sum\limits_{k\in {\cal K}_l} {\bf p}^{(1)}_k \otimes {\bf p}^{(2)}_k \otimes {\bf p}^{(3)}_k, \quad
  {\bf P}_{R_s}:=   \sum\limits_{k\in {\cal K}_s} {\bf p}^{(1)}_k \otimes {\bf p}^{(2)}_k \otimes {\bf p}^{(3)}_k.
\]
 This representation allows to reduce the calculation of the multiparticle interaction energy
of the many-particle system. Recall that the electrostatic interaction energy of $N$ charged particles 
is represented in the form
  \begin{equation}
 \label{inter_energy} 
 E_{N}= E_N(x_1,\dots,x_N)= \sum^{N}_{i=1} \sum^{N}_{j <i} \frac{z_i z_j}{\|x_i- x_j\|},  
 \end{equation}
 and it can be computed by direct summation in $O(N^2)$ operations.
The following statement is the modification of Lemma 4.2  in \cite{BKK_RS:17}, see \cite{Khor_Prosp:21} for more details. 

 \begin{lemma}\label{lem:InterEnergy} \cite{Khor_Prosp:21}
 Let the effective support of the short-range components in the reference
 potential ${\bf P}_R$ for the Newton kernel does not exceed the minimal distance between particles, $\sigma>0$.
 Then the interaction energy $E_N$ of the $N$-particle system  can be
 calculated by using only the long range part  in the tensor ${\bf P}$ representing on the grid the total potential sum,
 \begin{equation}\label{eqn:EnergyFree_Tensor}
 E_N = \frac{1}{2} 
 \sum\limits_{{j}=1}^N z_{j}({\bf P}_l({x}_{j}) - z_j {\bf P}_{R_l}(0))=
 \frac{1}{2}\langle {\bf z},{\bf p}_l \rangle - \frac{{\bf P}_{R_l}(0)}{2} \sum\limits_{{j}=1}^N z_j^2,
\end{equation}
in $O(R_l N)$ operations, where $R_l$ is the canonical rank of the long-range component in ${\bf P}$, ${\bf P}_l$.
\end{lemma}
Here ${\bf z}\in \mathbb{R}^N$ is a vector composed of all charges of the multi-particle systems, and 
${\bf p}_l\in \mathbb{R}^N$ is 
the vector of samples of the collective electrostatic long-range potential ${\bf P}_l$ in the nodes 
corresponding to particle locations. Thus, the term $\frac{1}{2}\langle {\bf z},{\bf p}_l \rangle$ 
denotes the "non-calibrated" interaction energy associated with the long-range tensor component ${\bf P}_l$, 
while ${\bf P}_{R_l}$ denotes the long-range 
part in the  tensor representing the  \emph{single reference Newton kernel}, and ${\bf P}_{R_l}(0)$ is its value at the 
origin.

Lemma \ref{lem:InterEnergy} indicates that the interaction energy does not depend on the short-range 
part in the collective potential, and this is the key point for the construction of 
energy preserving regularized numerical schemes for solving the basic equations in bio-molecular 
modeling by using low-rank tensor decompositions.

\subsection{Multilinear operations in RS tensor formats} \label{ssec:operations_RS_form}

In what follows,  we address the important question on how 
\emph{the basic multilinear operations } can be implemented in the RS tensor format by 
using the RHOSVD rank compression. The point is that various tensor operations arise in the course of commonly used numerical 
schemes and iterative algorithms which usually include many sums and products of functions 
as well as the actions of differential/integral 
operators, always  making the tensor structure of input data much more complicated requiring the robust rank reduction schemes. 

The other important aspect is related to the use of large (fine resolution)
discretization grids which is limited by the restriction on the  size of the full input tensors, $O(n^d)$ (curse of dimensionality), 
representing the discrete functions and operators to be approximated in low rank tensor format. 
Remarkably,  that tensor decomposition for special class of functions, which 
allow the sinc-quadrature approximation, can be performed on practically  indefinitely large grids because the storage and numerical costs
of such numerical schemes scale linearly in the univariate grid size, $O(d n)$. Hence, having constructed such low rank approximations for certain
set of ``reproducing'' radial functions, makes it possible to construct the low rank RS representation at linear complexity, $O(d n)$,
for the wide class of functions and operators by using the rank truncated multilinear operations.
The examples of such ``reproducing'' radial functions are commonly used in our computational practice.

First, consider the Hadamard product of two tensors ${\bf P}_R$ and ${\bf Q}_{R_1}$ corresponding to the pointwise 
product of two generating multivariate functions centered at the same point.
The RS representation of the product tensor is based on the observation that the long-range part of 
the Hadamard product of two tensors in RS-format is basically determined by the product of their long-range parts.
% \[
%  ({\bf P}_s + {\bf P}_l) \odot({\bf P}_s + {\bf P}_l),
% \]
% is basically determined by the product of their long-range parts, ${\bf P}_l \odot {\bf P}_l$.
% We postulate this statement in the following definition.
\begin{lemma}\label{lem:LRP_prod_tensor}
 Suppose that the RS representation (\ref{eqn:RS_repres}) of tensors ${\bf P}_R$ and ${\bf Q}_{R_1}$
 is constructed based on the sinc-quadrature CP  approximation (\ref{eqn:sinc_general}). 
 Then the long-range part of the Hadamard product  of these RS-tensors,  
 \[
 {\bf Z}=({\bf P}_s + {\bf P}_l) \odot({\bf Q}_s + {\bf Q}_l),
\]
can be represented by the product of their long-range parts, ${\bf Z}_l={\bf P}_l \odot {\bf Q}_l$, 
with the subsequent rank reduction.
Moreover, we have $rank({\bf Z}_l) \leq R_l Q_l$.
\end{lemma}
\begin{proof}
 We consider the case of collocation tensors and suppose that each skeleton vector in CP tensors ${\bf P}_R$ and ${\bf Q}_{R_1}$
 is given by the restriction of certain Gaussians   to the set of grid points. 
  Chose the arbitrary short-range components in ${\bf P}_R$ and some component in ${\bf Q}_{R_1}$, generated by Gaussians 
  $e^{-t_k x_\ell^2}$ and $e^{-t_m x_\ell^2}$, respectively.
 Then the effective support of the product of these two  terms becomes smaller than that for 
 each of the factors in view of the
 identity
 $e^{-t_k x_\ell^2}e^{-t_m x_\ell^2}= e^{-(t_k+t_m) x_\ell^2}$ considered for arbitrary $t_k, t_m >0$.
 This means that each term that includes the short-range multiple remains to be in the short range.
 Then the long range part in ${\bf Z}$ takes a form ${\bf Z}_l={\bf P}_l \odot {\bf Q}_l$ 
 with the subsequent rank reduction.
\end{proof}

The sums of several tensors in RS format can be easily split into short- and long-range parts 
by grouping the respective components
in the summands. The other important operation is the operator-function product in RS tensor format (see the example in 
\cite{BKhor_Dirac:18} related to the action of Laplacian with the singular Newton kernel resulting in the RS 
decomposition of the Dirac delta). This topic will be considered in detail elsewhere.

\subsection{Representing the Slater potential in RS tensor format} \label{ssec:Slater_RS_form}

In what follows, we consider the RS-canonical tensor format for the rank-structured representation of the Slater function 
\[
 G(x)=e^{-\lambda \|x\|}, \quad \lambda\in \mathbb{R}_+,
\]
 which  has the principal significance in electronic structure calculations (say, based on the Hartree-Fock equation) since
it represents the cusp behavior of electron density in the local vicinity of nuclei.
This function (or its approximation) is considered as the best candidate to
be used as the localized basis function for atomic orbitals basis sets.
Another direction is related to the construction of the accurate low-rank global interpolant for big scattered data
to be considered in the next section.
In this way, we calculate the data adaptive basis set living on the fine Cartesian grid in the region of target data.
The main challenge, however, is due to the presence of point singularities which are hard to approximate 
in the problem independent polynomial or trigonometric basis sets.

%Hence, our goal is the construction of low-rank RS approximation to the Slater function.
%in the form of sum of Gaussians.

The construction of low-rank RS approximation to the Slater function is based on 
the generalized Laplace transform representation for the Slater function written in the 
form $G(\rho)={e^{-2 \sqrt{\alpha \rho}}}$, $\rho(x)=  x_{1}^2+...+ x_{d}^2$,
reads
\begin{equation*} \label{Slater Gauss tr}
G(\rho)={e^{-2 \sqrt{\alpha \rho}}}=
\frac{\sqrt{\alpha}}{\sqrt{\pi}}\int_{\mathbb{R}_+} \tau^{-3/2}
\exp(- \alpha/\tau - \rho \tau) d \tau,
%\quad \kappa\in [0,\infty),
\end{equation*}
which corresponds to the choice 
$\widehat{G}(\tau) =\frac{\sqrt{\alpha}}{\sqrt{\pi}} \tau^{-3/2} e^{-\alpha/\tau}$ in the  canonical form of the Laplace 
transform representation for $G(\rho)$,
\begin{equation} \label{Slater_Laplace_tr}
 G(\rho)=\int_{\mathbb{R}_{+}}\widehat{G}(\tau)e^{-\rho\tau}d\tau.
\end{equation}
Denote by ${\bf G}_R$ the rank-$R$ canonical approximation to the function $G(\rho)$ discretized on 
the $n\times n\times n$ Cartesian grid.
\begin{lemma}\label{lem:sincSlater} (\cite{Khor1:06})
For given threshold $\varepsilon>0$  let $\rho\in [1,A]$. Then the $(2M+1)$-term 
sinc-quadrature approximation of the integral in (\ref{Slater_Laplace_tr}) with
\begin{equation*} \label{Slater_sinc_error}
M=O(|\log \varepsilon|(|\log \varepsilon|+ \log A )),
\end{equation*}
ensures the max-error of the order of $O(\varepsilon)$ for 
the corresponding rank-$(2M+1)$ CP approximation ${\bf G}_R$ to the tensor ${\bf G}$.
\end{lemma}

\begin{figure}[htb]
\centering
\includegraphics[width=7.0cm]{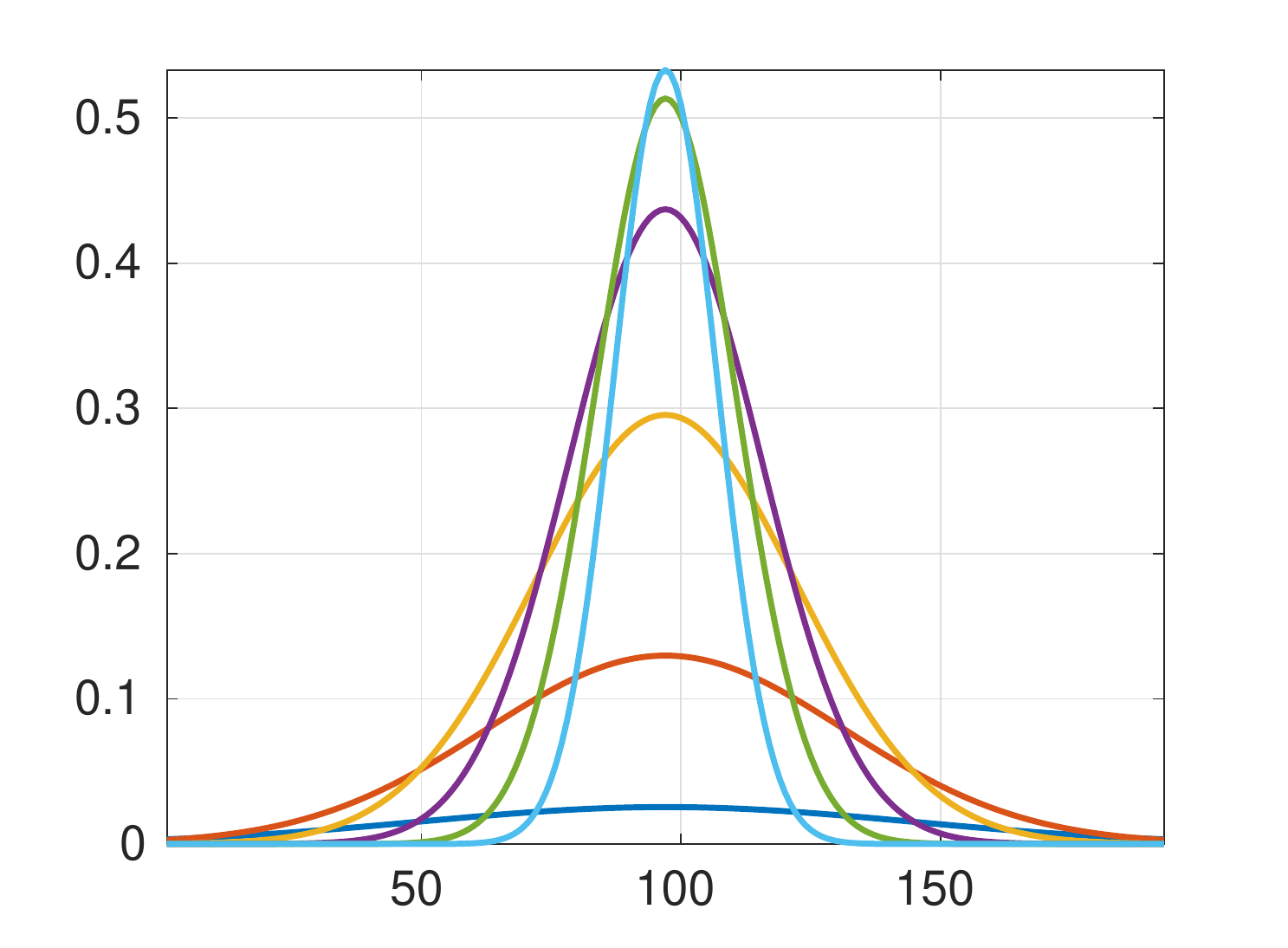} \quad
\includegraphics[width=7.0cm]{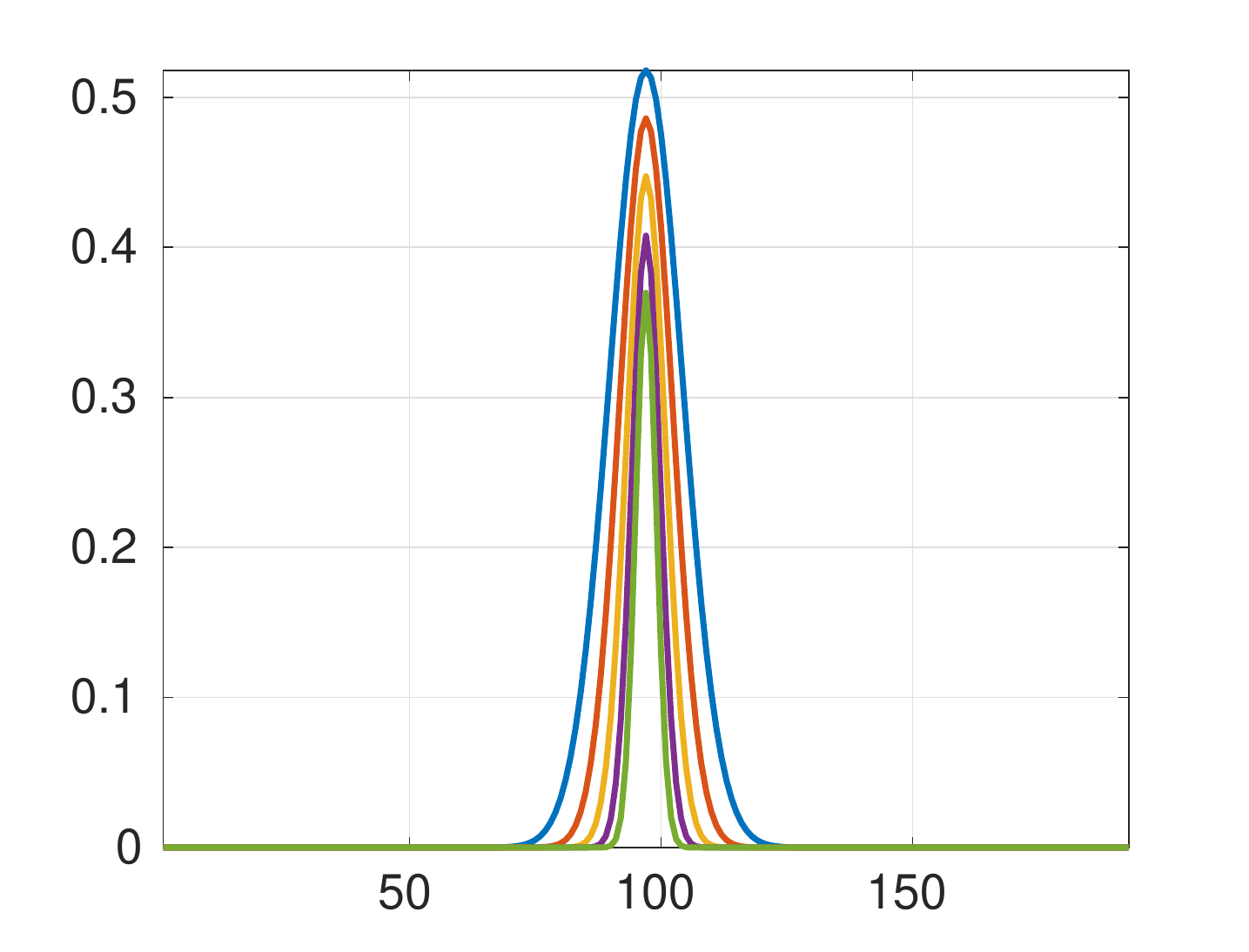}
 \caption{\small Long-range (left) and short-range (right, logarithmic scale) canonical vectors 
 for the Slater function with the grid size $n=1024$, $R=24, R_l=6$, $\lambda=1$.}
\label{fig:1024_rl6_Slat}
\end{figure}

Figure \ref{fig:1024_rl6_Slat} illustrates the RS splitting for the tensor 
${\bf G}_R= {\bf G}_{R_l} + {\bf G}_{R_s}$ 
representing the Slater potential $G(x)= e^{- \lambda \|x\|}$, $\lambda =1$, discretized 
on the $n\times n\times n$ grid with $n=1024$.  The rank parameters are chosen by $R=24, R_l=6$ and $R_s=18$.
Notice that for this radial function the long-range part (Figure \ref{fig:1024_rl6_Slat}, left) 
includes much less canonical 
vectors comparing with the case of Newton kernel.
This anticipates the smaller total canonical rank for the long-range part 
in the large sum of Slater-like potentials arising, for example, in the representation 
of molecular orbitals and the electron density 
in electronic structure calculations. 
For instance, the wave function for the Hydrogen atom is given by the Slater function $e^{-\mu \|x\|}$. 
In the following section, we consider the application of RS tensor format
to interpolation of scattered data in $\mathbb{R}^d$.

 \subsection{Application of RHOSVD to scattered data modeling} \label{sec:Apll2Data}

 In scattered data modeling the problem is in a low parametric approximation of
multi-variate functions $f:\mathbb{R}^d \to \mathbb{R}$ %($d \geq 2$) 
by sampling at a finite set  ${\cal X}=\{x_1,\ldots,x_N\}\subset \mathbb{R}^d$
of piecewise distinct points. 
Here the function $f$ might be the surface of a solid body, the solution of a PDE, 
many-body potential field, multi-parametric characteristics of physical systems, 
or some other multi-dimensional data, etc. 

 Traditional ways of recovering $f$ from a sampling vector
$f_{|{\cal X}}=(f(x_1),\ldots,f(x_N))$ is the 
constructing  a functional interpolant $P_N:\mathbb{R}^d \to \mathbb{R}$ such that 
$P_{N|{\cal X}}=f_{|{\cal X}}=:{\bf f}\in \mathbb{R}^N$, i.e.,
\begin{equation}\label{eqn:Interp_col}
 P_N(x_j)=f(x_j), \quad \forall\; 1\leq j \leq N. 
\end{equation}
 
Using {radial basis (RB) functions} one can find interpolants $P_N$ in the form
\begin{equation}\label{eqn:RBF_interp}
 P_N(x) = \sum_{j=1}^N c_j p(\|x-x_j\|) + Q(x), 
 \quad Q \; \mbox{is some smooth function, say, polynomial},
\end{equation}
where
$p=p(r):[0,\infty)\to \mathbb{R}$ is a fixed RB function, and $r=\|\cdot \|$ is
the Euclidean norm on $\mathbb{R}^d$. In further discussion we set $Q(x)=0$.
For example, the following RB functions are commonly used
$$
 p=r^\nu, \quad (1+r^2)^\nu,\; (\nu\in \mathbb{R}),
 \quad \exp(-r^2), \quad \exp(-\lambda r), \quad r^2 \log (r).
$$
The other examples of RB functions are defined by Green's kernels or by the class
of Mat{\'e}rn functions \cite{LiKeKKMa:19}.

 We discuss the following computational tasks (A) and (B). 
\begin{itemize}
 \item[(A)] Fixed coefficient vector ${\bf c}=(c_1,\ldots,c_N)^T\in \mathbb{R}^N$,
efficiently representing the interpolant $P_N(x)$
on the fine tensor grid in $\mathbb{R}^d$ providing \\ 
(a) $O(1)$-fast point evaluation of $P_N$ in the computational volume $\Omega$, \\ 
(b) computation of various integral-differential operations on that interpolant 
(say, {gradients, scalar products, convolution integrals, etc.}).
\item[(B)] Finding the coefficient vector ${\bf c}$  that solves the 
 interpolation problem (\ref{eqn:Interp_col}) in the case of large number $N$.
\end{itemize}

 Problem {(A)} exactly fits the RS tensor framework so that the RS tensor approximation 
 solves the problem with low computational costs provided that the sum of long-range parts of 
 the interpolating functions can be easily approximated in the low rank CP tensor format. 
 We consider the case of interpolation by Slater functions $\exp(-\lambda r)$  in the more detail.

 Problem {(B):} Suppose that we use some favorable preconditioned iteration for 
solving coefficient vector ${\bf c}=(c_1,\ldots,c_N)^T$, 
\begin{equation}\label{eqn:RBF_Syst}
 A_{p,{\cal X}} {\bf c}= {\bf f}, \quad \mbox{with}\quad 
 A_{p,{\cal X}}=A_{p,{\cal X}}^T =[p(\|x_i-x_j\|)]_{1\leq i,j\leq N}\in \mathbb{R}^{N\times N},
\end{equation}
with the  distance dependent symmetric system matrix $A_{p,{\cal X}}$. 
We assume 
 ${\cal X}=\Omega_h$ be the $n^{\otimes d}$-set of grid-points located on
 tensor grid, i.e., $N=n^d$.
Introduce the $d$-tuple multi-index $i \mapsto {\bf i}=(i_1,\ldots,i_d)$, and 
$j \mapsto {\bf j}=(j_1,\ldots,j_d)$ and reshape  $ A_{p,{\cal X}}$ into the tensor form
$$ 
A_{p,{\cal X}}\mapsto {\bf A}=[a(i_1,j_1,\ldots,i_d,j_d)]\in 
  {\bigotimes}_{\ell=1}^d \mathbb{R}^{n\times n}, %\quad {\bf A}={\bf A}_{R_s} + {\bf A}_{R_l}.
$$
which can be decomposed by using  the RS based splitting
\[
 {\bf A}={\bf A}_{R_s} + {\bf A}_{R_l},
\]
generated by the RS representation of the weighted potential sum in (\ref{eqn:RBF_interp}).
Here ${\bf A}_{R_s}$ is (almost) diagonal matrix, while
 ${\bf A}_{R_l}= {\sum}_{k=1}^{R_l} A^{(1)}_k \otimes \cdots \otimes A^{(d)}_k$
 is the low Kronecker rank matrix.
This implies a bound on the storage, $O(N + d R_l n)$, 
 and ensures  a fast matrix-vector multiplication.
 Introducing the additional rank-structured representation in ${\bf c}$, the solution of 
 (\ref{eqn:RBF_Syst}) can be further simplified. 
 
 The above approach can be applied to the data sparse representation for the 
 class of large covariance matrices in the spacial statistics, see for example  \cite{Matern:86,LiKeKKMa:19}.

In application of tensor methods to data modeling  (see \S\ref{sec:Apll2Data}) 
we consider the interpolation of 3D scattered data by a large sum of Slater functions
\begin{equation}\label{eqn:Slater_interp}
 G_N(x) = \sum_{j=1}^N c_j e^{-\lambda \|x-x_j\|}, \quad \lambda >0.
 \end{equation}
Given the coefficients $c_j$, we address the question how to efficiently represent the interpolant $G_N(x)$
on fine Cartesian grid in $\mathbb{R}^3$ by using the low-rank (i.e., low-parametric) CP tensor format,
such that each value on the grid can be calculated in $O(1)$ operations. The main problem is that 
the generating Slater function $e^{-\lambda \|x\|}$ has the cusp at the origin so that the 
considered interpolant has very low regularity. As result, the tensor rank
of the function $G_N(x)$ in (\ref{eqn:Slater_interp}) discretized on a large $n\times n \times n$ grid 
increases proportionally 
to the number $N$ of sampling points $x_j$, which in general may be very large. 
Hence the straightforward  tensor approximation of $G_N(x)$ does not work in this case.

 The generating Slater radial function can be proven to have the low-rank RS canonical tensor decomposition 
 by using the sinc-approximation method, see \S\ref{ssec:Theory_RS_sinc}.

To complete this section, we present the numerical example demonstrating the application of RS tensor 
representation to scattered data modeling in $\mathbb{R}^{3}$. 
We denote by ${\bf G}_R\in \mathbb{R}^{n\otimes 3}$ the rank-$R$ 
CP tensor approximation of the reference Slater potential 
$e^{-\lambda \|x\|}$ discretized on  $n\times n \times n$ grid $\Omega_n$, and
introduce its RS splitting ${\bf G}_R= {\bf G}_{R_l} + {\bf G}_{R_s}$,
with $R_l+R_s =R$. Here $R_l\approx R/2$ is the rank parameter of the 
long-range part in ${\bf G}_R$. Assume that all measurement points 
$x_j$ in (\ref{eqn:Slater_interp}) are located on the discretization 
grid $\Omega_n$, then the tensor representation of the long-range part 
of the total interpolant
$P_N$ can be obtained as the sum of the properly replicated reference 
potential ${\bf G}_{l}$, via the shift-and-windowing transform ${\cal W}_j$, $j=1,\ldots,N$,
\begin{equation}\label{eqn:Slater_interp_multi}
 {\bf G}_{N,l}= \sum_{j=1}^N c_j {\bf G}_{l,j}, \quad {\bf G}_{l,j} = {\cal W}_j {\bf G}_{l},
\end{equation}
that includes about $N \, R_l$ terms. For large number of measurement points, $N$, the rank reduction is ubiquitous.
 
 It can be proven (by slight modification of arguments in \cite{BKK_RS:17}) that both the CP and Tucker ranks
 of the $N$-term sum in (\ref{eqn:Slater_interp_multi}) depend only logarithmically (but not linearly) on $N$.
  
 \begin{proposition}\label{prop:Slater_LR_rank} (Uniform rank bounds for the long-range part in the Slater interpolant).
Let the long-range part ${\bf G}_{N,l}$ in the total Slater interpolant in (\ref{eqn:Slater_interp_multi})
be composed of those terms in (\ref{eqn:sinc_Newt}) which satisfy the relation $t_k \leq 1$,
where $M=O(\log^2\varepsilon)$.
Then the total $\varepsilon$-rank ${\bf r}_0$ of the Tucker approximation to the canonical tensor sum ${\bf G}_{N,l}$
is bounded by
\begin{equation}\label{eqn:Rank_LongR}
 |{\bf r}_0|:=rank_{Tuck}({\bf G}_{N,l})=C\, b \,\log^{3/2} (\log (N/\varepsilon)),
\end{equation}
where the constant $C$ does not depend on the number of particles $N$, 
as well as on the size of the computational box, $[-b,b]^3$.
 \end{proposition}
\begin{proof}
 (Sketch) The main argument of the proof is based on the fact that the grid function ${\bf G}_{N,l}$ has the band-limited 
 Fourier image, such that the frequency interval depends weakly (logarithmically) on $N$. Then we represent all Gaussians 
 in the truncated Fourier basis and make the summation in the fixed set of orthogonal trigonometric basis functions, 
 which defines the orthogonal Tucker 
 representation with controllable rank parameter.
\end{proof}
The numerical illustrations below demonstrate the CP rank by RHOSVD decomposition of the long-range part ${\bf G}_{N,l}$ in 
the multi-point tensor interpolant via Slater functions. 

 \begin{table}[htb]
\begin{center}\footnotesize 
\begin{tabular}
[c]{|r|r|r|r|r||}%
 \hline
 $L_1\times L_2\times L_3$ & $N$  & Tucker ranks  & $R_{ini}$ & $R_{comp}$ \\
  \hline
 $4\times 4\times 4$  & 64 &  $13\times 13 \times 13$ & 192  & 56      \\ 
 \hline
  $6\times 6\times 6$  & 216 & $ 15\times 15 \times 15$ & 649  & 95   \\ 
 \hline
 $8\times 8\times 8$  & 512 & $19\times 19 \times 19$ & 1536  & 131  \\ 
 \hline
 $16\times 16\times 8$  & 2048 & $32\times 32 \times 19$ & 6141  & 253   \\ 
 \hline
 $16\times 16\times 16$  & 4096 & $32\times 32 \times 32$ & 12288  & 380   \\ 
 \hline
 \end{tabular}
\caption{\small Reduced ranks for the case of random amplitudes.
$\varepsilon_{Tuck} = 10^{-3}$, $\varepsilon_{T2C} = 10^{-5}$, $R_L =3$, $R=5$.}
\label{Tab:random_ranks}
\end{center}
\end{table} 
  Now, we generate a tensor composed of a sum of Slater functions, 
 discretized by collocation over $n^{\otimes 3}$ representation grid with $n=384$, 
 and placed in the nodes of a sampling $L_1\times L_2 \times L_3$ lattice 
 with randomly chosen weights $c_j$ in the interval $c_j \in [-5,5]$ for every node. 
 Every  single Slater function is generated as a canonical tensor by using sinc-quadratures for the approximation of 
 the related Laplace transform.
  Table \ref{Tab:random_ranks} shows ranks of the long-range part of this tensor composed of 
  Slater potentials located in the nodes of the lattices of increasing size. 
  $N$ indicates the number of nodes,   
  while $R_{ini}$ and $R_{comp}$ are the initial and compressed canonical
  ranks of the resulting long-range part tensor, respectively. Tucker ranks correspond to the 
  ranks in the canonical-to-Tucker decomposition step.
  Threshold values for the Slater potential generator is 
  $\varepsilon_N=10^{-5}$, while the tolerance thresholds for the 
  rank reduction procedure are given by 
  $\varepsilon_{Tuck}=10^{-3}$ and $\varepsilon_{T2C}=10^{-5}$.
  We observe that the ranks of the long-range part of the potential increase only slightly 
  in the size of the 3D sampling lattice, $N$.

Figure \ref{fig:Slat_multi_randon_12x12x4} demonstrates the
full-, short- and long-range components of the  multi-Slater tensor constructed by   
   the weighted sum of Slater functions with {\it randomly} chosen weights $c_j$ in the 
  interval $c_j \in [-5,5]$. The positions of the generating nodes are located on the $12\times 12 \times 4$ 3D lattice.
  The parameters of the tensor interpolant  are set up as follows: $\lambda=0.5$, 
  the representation grid is of size $n^{\otimes 3}$ with $n=384$, $R=8, R_l=3$ and 
  the number of samples $N=576$. (Figures zoom a part of the grid.)
  The initial CP rank of the sum of $N_0$ interpolating Slater potentials is about $4468$.
%   while   the rank of RHOSVD for the long-range part of the composite tensor ${\bf G}_{N,l}$ is   $252$. 
%   The rank truncation threshold is $\varepsilon = 10^{-5}$.
   Middle and right pictures show the long- and short-range parts of the 
  composite tensor, respectuvely. The initial rank of the canonical tensor 
 representing the  long-range part is equal to $R_L=2304$, which is reduced
 by the C2C procedure via RHOSVD to $R_{cc}=71$. 
 The rank truncation threshold is $\varepsilon = 10^{-3}$.

\begin{figure}[htb]
\centering
\includegraphics[width=5.3cm]{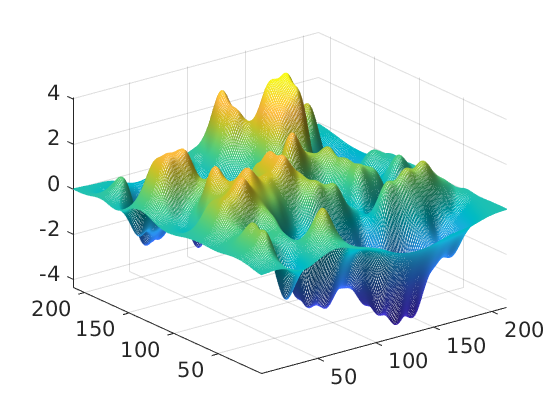} 
\includegraphics[width=5.3cm]{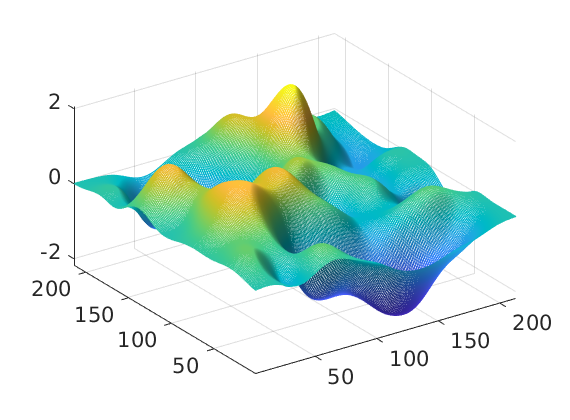}
\includegraphics[width=5.3cm]{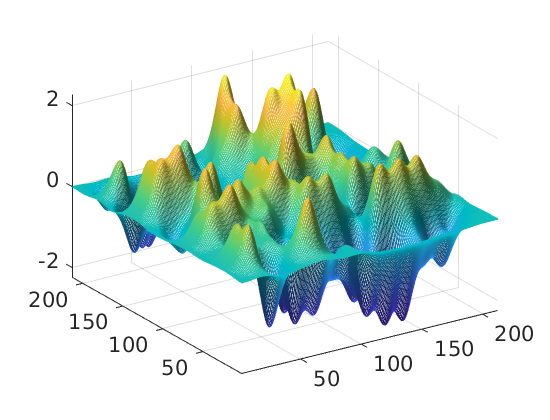} 
 \caption{\small Full-, long- and short-range components of the 
  multi-Slater tensor.   Slater kernels with $\lambda=0.5$ and with random amplitudes in the range of $[-5,5] $ are placed 
  in the nodes of a $12 \times 12\times 4 $ lattice using  
  3D grid of size $n^{\otimes 3}$ with $n=384$, $R=8, R_l=3$ and   the number of nodes $N=576$.}
\label{fig:Slat_multi_randon_12x12x4}
\end{figure}
 
  Figure \ref{fig:Slat_funct_12x12x4} and Table \ref{Tab:func_ranks} demonstrate the decomposition 
  of the multi-Slater tensor with the amplitudes in the nodes  modulated by the function
  of the (x,y,z)-coordinates 
 \begin{equation}   \label{eq:func2}
  F(x,y,z) = c_1 \cos(x+2y+4z) \exp(-c_2\sqrt{x^2 +2y^2 +4z^2}),
 \end{equation}
 with $c_1=6$ and $c_2=0.1$.

  Next, we generate a tensor composed of a sum of discretized Slater functions 
 on a sampling lattice $L_1\times L_2 \times L_3$, living on  
  3D representation grid of size $n^{\otimes 3}$ with $n=232$.
 The amplitudes of the individual Slater functions are
   modulated by a function of $x,y,z$-coordinates (\ref{eq:func2}) in every node of the lattice.
  Table \ref{Tab:func_ranks} shows rank of the long-range part of this multi-Slater tensor  
  with respect to the  increasing size of the lattice. $N=L_1\, L_2 \, L_3$ is the number of nodes, and $R_{ini}$ 
  and $R_{comp}$ are the initial and compressed canonical
  ranks, respectively. Tucker ranks are shown at the canonical-to-Tucker decomposition step.  
  Threshold values for the Slater potential generation is 
  $\varepsilon_N=10^{-5}$, the thresholds for the canonical-to-canonical
  rank reduction procedure are given by
  $\varepsilon_{Tuck}=10^{-3}$ and $\varepsilon_{T2C}=10^{-5}$. 
  Table \ref{Tab:func_ranks} demonstrates the very moderate icrease of the 
  reduced rank in the long-range part of the Slater potential sum on the size of the 3D sampling lattice.
       
 \begin{table}[htb]
\begin{center}\footnotesize 
\begin{tabular}
[c]{|r|r|r|r|r||}%
 \hline
 $L_1\times L_2\times L_3$ & $N$   & Tucker ranks &  $R_{ini}$ & $R_{comp}$  \\
  \hline
 $4\times 4\times 4$  &  64 & $7\times 7\times 7 $ & 256  & 25      \\ 
 \hline
  $6\times 6\times 6$  & 216 & $7\times 7\times 7 $ & 648  & 23   \\ 
 \hline
 $8\times 8\times 8$  & 512 & $7\times 7\times 7 $ & 1536  & 30   \\ 
 \hline
 $16\times 16\times 8$  & 2048 & $10\times 9\times 6 $ & 6144  & 47   \\ 
 \hline
 $16\times 16\times 16$  & 4096 & $10\times 9\times 8 $ & 12288  & 55   \\ 
 \hline
 \end{tabular}
\caption{\small Reduced canonical ranks for the case of functional amplitudes. 
$R_L =3$, $R_L =5$, $\varepsilon_N=10^{-5}$, $\varepsilon_{Tuck}=10^{-3}$, $\varepsilon_{T2C}=10^{-5}$.}
\label{Tab:func_ranks}
\end{center}
\end{table}  
 Figure \ref{fig:Slat_funct_12x12x4} demonstrates the full-, long-  and short-range components of the 
  multi-Slater tensor.   Slater kernels with $\lambda=0.5$ and with the  amplitudes 
  modulated by the function (\ref{eq:func2}) 
  of the $(x,y,z)$-coordinates are places on the nodes of a $12 \times 12\times 4 $ sampling lattice, living on  
  3D grid of size $n^{\otimes 3}$ with $n=384$, $R=8, R_l=3$, and with  the number of sampling nodes $N=576$.
 \begin{figure}[htb]
\centering
\includegraphics[width=5.3cm]{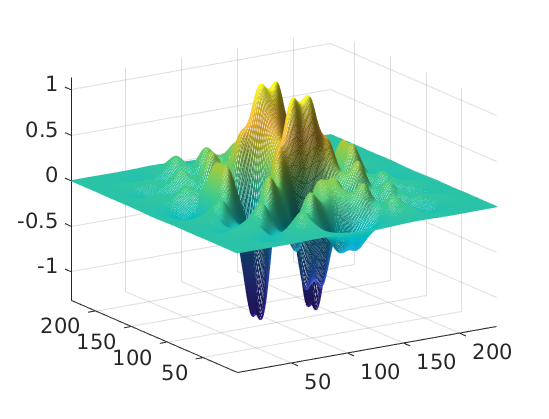} 
\includegraphics[width=5.3cm]{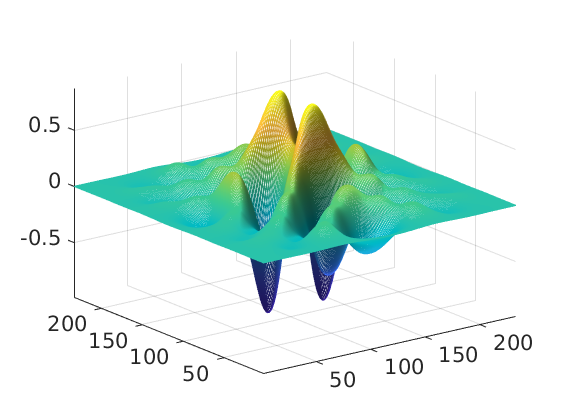}
\includegraphics[width=5.3cm]{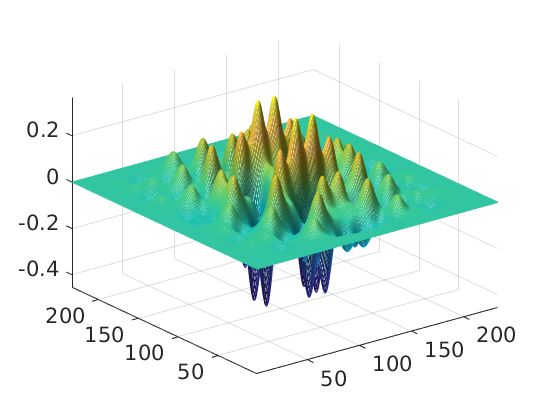} 
 \caption{\small Full-, long- and short-range components of the 
  multi-Slater tensor.   Slater kernels with $\lambda=0.5$ and with  amplitudes modulated 
  by the function (\ref{eq:func2}) using the nodes of a $12 \times 12\times 4 $ lattice on  
  3D grid of size $n^{\otimes 3}$ with $n=384$, $R=8, R_l=3$ and   the number of nodes $N=576$.}
\label{fig:Slat_funct_12x12x4}
\end{figure}

%   \begin{figure}[htb]
% \centering
% \includegraphics[width=5.3cm]{func4_c2c_err_12x12x4_r2.png} 
% %\includegraphics[width=5.3cm]{Figs_MM/Slat_12x12x4_RS_err.png} 
%  \caption{\small The  RHOSVD approximation error of the long-range part in the functional tensor ${\bf G}_{N,l}$.}
% \label{fig:Slat_funct_long_12x12x4_error}
% \end{figure}
%  
%  \cb
%  Figure \ref{fig:Slat_funct_12x12x4} demonstrates the full-, short- and long-range components of the 
%   multi-Slater tensor.   Slater kernels with $\lambda=0.5$ and with the  amplitudes modulated by the function (\ref{eq:func2}) 
%   of the $(x,y,z)$-coordinates of the nodes of a $12 \times 12\times 4 $ lattice on  
%   3D grid of size $n^{\otimes 3}$ with $n=232$, $R=8, R_l=3$ and with  the number of nodes $N_0=576$.  
%   
%   \cn

 \section{Representing Green's kernels in tensor format}\label{sec:Fundam_sol}
 
 In this section we demonstrate how the RHOSVD can be applied for the efficient tensor decomposition  of various 
 singular radial functions
 composed by polynomial expansions of a few reference potentials already precomputed in the low-rank tensor format.
 Given the low-rank CP tensor ${\bf A}$ further considered as a reference tensor, 
 the low rank representation of the tensor-valued polynomial function
 \[
  {P}({\bf A})= a_0 {\bf I} + a_1 {\bf A} + a_2 {\bf A}^2 + \cdots +a_n {\bf A}^n,
 \]
 where the multiplication of tensors is understood in the sense of pointwise Hadamard product,
can be calculated via $n$-times application of the RHOSVD by using the Horner scheme in the form
\[
 P(x) = a_0  + x(a_1 + x(a_2 + x(a_3 + \ldots x(a_{n-1} +a_n x) \ldots ))).
\]
Similar scheme can be also applied in the case of multivariate polynomials.

For examples considered in this section we make use of the discretized Slater 
$e^{-\|x\|}$ and Newton $\frac{1}{\|x\|}$, 
$x\in \mathbb{R}^d$, kernels as the reference tensors. The following statement 
was proven in \cite{HaKhtens:04I,Khor1:06}
(see also Lemma \ref{lem:sincSlater}). 
\begin{proposition} \label{prop:tensor_green}
The discretized over $n^{\otimes d}$-grid radial functions $e^{-\|x\|}$ and $\frac{1}{\|x\|}$, $x\in \mathbb{R}^d$, 
included in representation of various 
Green kernels and fundamental solutions for elliptic operators with constant coefficients, both
allow the low-rank CP tensor approximation. The corresponding rank-$R$ representations can be calculated in 
$O(d R n)$ operations without precomputing and storage of the target tensor in the full (entry-wise) format.
\end{proposition}

Tensor decomposition for discretized singular  kernels such as $\|x\|$, $\frac{1}{\|x\|^m}$, $m \geq 2$, 
and $e^{-\kappa \|x\|}/\|x\|$,
can be now calculated by applying the RHOSVD to polynomial combinations of the reference potentials 
as in Proposition \ref{prop:tensor_green}. The most important benefit of the presented techniques is the opportunity 
to compute the rank-$R$ tensor approximations without precomputing and storage of the target tensor in the 
full format tensor.
  
  In what follows, we present the particular examples of singular kernels in $\mathbb{R}^d$ which 
  can be treated by the above  
presented  techniques.
  Consider the fundamental solution of the advection-diffusion operator ${\cal L}_d$ with constant coefficients 
in $\mathbb{R}^d$
\[
 {\cal L}_d = -\Delta + 2\bar{b} \cdot \nabla + \kappa^2, \quad \bar{b}\in \mathbb{C}^d, \quad 
 \kappa\in \mathbb{C}.
\]
If $\kappa^2 + |\bar{b}|^2 = 0$, then for $d\geq 3$ it holds
\[
 \eta_0(x) = \frac{1}{(d-2) \omega_d} \frac{e^{\langle \bar{b},x \rangle }}{\|x\|^{d-2}},
\]
where $\omega_d$ is the surface area of the unit sphere in $\mathbb{R}^d$.
Notice that the radial function $\frac{1}{\|x\|^{d-2}}$ for $d\geq 3$ allows the RS decomposition of the 
corresponding discrete tensor representation based on the sinc quadrature approximation, which implies
the RS representation of the kernel function $\eta_0(x)$, 
since the function $e^{\langle \bar{b},x \rangle }$ is already separable.
From computational point of view, both the CP and RS canonical decompositions of 
discretized kernels $\frac{1}{\|x\|^{d-2}}$
can be computed by successive application of RHOSVD approximation to the products of canonical tensors for 
the discretized Newton potential $\frac{1}{\|x\|}$.

In the particular case $\bar{b}=0$, we obtain the fundamental solution of the operator
${\cal L}_3 = -\Delta + \kappa^2$ for $d=3$, also known as the Yukawa 
(for $\kappa\in \mathbb{R}_+$) or Helmholtz (for $\kappa\in \mathbb{C}$) Green kernels
\[
\eta_\lambda(x) =  \frac{1}{4\pi} e^{-\kappa \|x\|}/\|x\|, \quad x\in \mathbb{R}^3.
\]
 In the case of Yukawa  kernel the tensor representations
by using Gaussian sums are considered in  
\cite{Khor_book:18,Khor_bookQC:18}, see also references therein.
%and implemented in \cite{BeHaKh:08}.

The  Helmholtz equation with $\mbox{Im}\, \kappa >0$ (corresponds to the diffraction potentials) 
arises in problems of acoustics, electro-magnetics and optics.
We refer to \cite{MazSch:book:07} for the detailed discussion of this class of fundamental solutions. 
Fast algorithms for the oscillating Helmholtz kernel have been considered in  \cite{Khor_book:18}. 
However, in this case the construction of the RS tensor decomposition remains an open question.

In the  case of 3D  biharmonic operator ${\cal L}= \Delta^2$ the fundamental solution reads as
\[
 p(\|x\|)=- \frac{1}{8\pi} \|x\|, \quad x\in \mathbb{R}^3.
\]

% They can be used for the solution of corresponding 
% equations in non-homogeneous  media and with the right-hand side composed of many pointwise 
% singular densities (forces).

The hydrodynamic potentials correspond to  the classical Stokes  operator
\[
 \nu \Delta {u} -\mbox{grad} \, p = { f}, \quad \mbox{div}\, {u} =0,
\]
where ${u}$ is the velocity field, $p$ denotes the pressure, and $\nu$ is the constant viscosity coefficient.
The solution of the Stokes problem in $\mathbb{R}^3$ can be expressed by the hydrodynamic potentials
\begin{equation}
 u_k (x) = \int\limits_{\mathbb{R}^3} \sum\limits^3_{\ell=1}
 \Psi_{k\ell} (x -{y}) f_\ell ({y}) \mbox{d}{y},\quad 
 p(x) = \int\limits_{\mathbb{R}^3} 
 \langle \Theta (x -{y}),{ f} ) ({y})\rangle \mbox{d}{y}
 \end{equation}
 with the fundamental solution
 \begin{equation}
  \Psi_{k\ell} (x) =\frac{1}{8\pi \nu} 
  \left(\frac{\delta_{k\ell}}{\|x\|} + \frac{x_k x_\ell}{\|x\|^3}\right), \quad
  \Theta (x)= \frac{x}{4\pi \|x\|^3}, \quad x\in \mathbb{R}^3.
 \end{equation}
 The existence of the low-rank RS tensor representation for  
the hydrodynamic potential is based on the same argument as in Remark \ref{prop:tensor_green}.
In turn, in the case of biharmonic fundamental solution we use the identity 
\[
 \|x\|= \frac{\|x\|^2}{\|x\|},
\]
where the nominator has the separation rank equals to $d$. The latter representation can be also applied for 
calculation of the respective tensor approximations.

Here we demonstrate how the application of RHOSVD allows to easily compute the low rank 
Tucker/CP approximation of the 
discretized singular potential $\frac{1}{\|x\|^3}$, $x \in \mathbb{R}^{3}$, 
as well as the respective RS-representation, 
having at hand the RS representation of the tensor ${\bf P}\in \mathbb{R}^{n\otimes 3}$ 
discretizing the Newton kernel.
In this example, we use the discretization of $\frac{1}{\|x\|^3}$ in the form 
\[
 {\bf P}^{(3)} = {\bf P}\odot {\bf P}\odot {\bf P},
\]
where by ${\bf P}^{(3)}$ we denotes the collocation projection discretization of $\frac{1}{\|x\|^3}$.
The low rank Tucker/CP tensor approximation to ${\bf P}^{(3)}$ can be computed by the direct application of 
the RHOSVD to the above product type representation. The RS representation of ${\bf P}^{(3)}$ is calculated 
based on Lemma \ref{lem:LRP_prod_tensor}.
Given the  RS-representation (\ref{eqn:RS_repres}) of the discretized Newton kernel, ${\bf P}_R$, we define  
the low rank CP approximation to the discretized singular part in the hydrodynamic potential ${\bf P}^{(3)}$ by
\[
 {\bf P}^{(3)}_{R'} = {\bf P}_R\odot {\bf P}_R \odot {\bf P}_R.
\]
In view of Lemma \ref{lem:LRP_prod_tensor}, the long range part of RS decomposition of ${\bf P}^{(3)}_{R'}$, can 
be computed by RHOSVD approximation to the following Hadamard product of tensors,
\[
 {\bf P}^{(3)}_{R'_l}={\bf P}_{R_l}\odot {\bf P}_{R_l} \odot {\bf P}_{R_l}.
\]
Figure \ref{fig:hydrodin_full_long_n257_error} visualizes the tensor ${\bf P}^{(3)}_{R'}$ as well 
as its long range part ${\bf P}^{(3)}_{R'_l}$.
 \begin{figure}[htb]
\centering
\includegraphics[width=5.3cm]{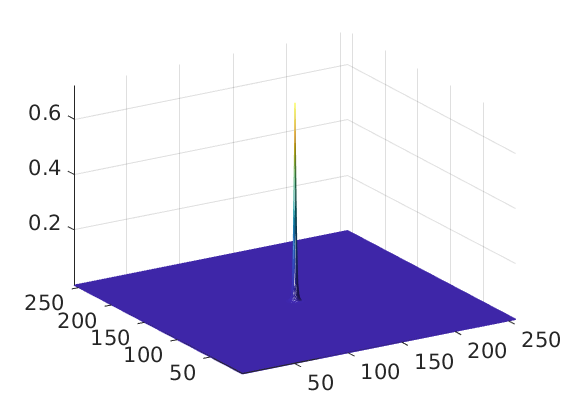} 
\includegraphics[width=5.3cm]{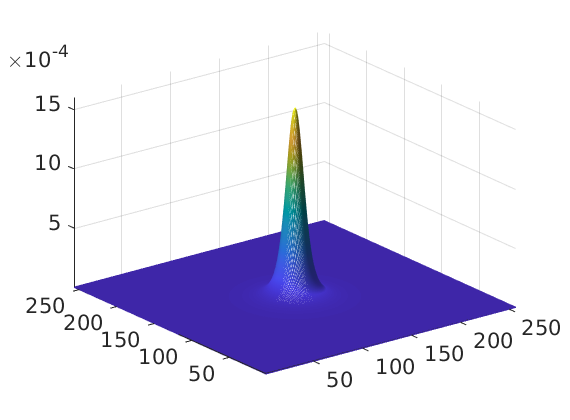} 
 \caption{\small RHOSVD approximation of the discretized cubic potential $\frac{1}{\|x\|^3}$ 
 and its long-range part.}
\label{fig:hydrodin_full_long_n257_error}
\end{figure}
The potentials are discretized on $n \times n \times n$ Cartesian grid with $n=257$, 
the rank truncation threshold is chosen 
for $\varepsilon=10^{-5}$. The CP rank of the Newton kernel is equal to $R=19$, while we set $R_l=10$, thus
resulting in the initial ranks $6859$ and $10^3$ for RHOSVD decomposition 
of ${\bf P}^{(3)}_{R'}$ and ${\bf P}^{(3)}_{R'_l}$, respectively. 
The RHOSVD decomposition reduces the large rank parameters to 
$R'= 122$ (the Tucker rank is $r=13$) and ${R'_l}=58$ (the Tucker rank is $r=8$), correspondingly.

\section{RHOSVD for rank reduction in 3D elliptic problem solvers}
\label{sec:solvers}

 Efficient rank reduction procedure based on the RHOSVD is a prerequisite for the development
 of the tensor-structured solvers for the three-dimensional elliptic 
 problem, which reduce  the computational complexity to
   almost linear scale, $O(nR)$, contrary to usual $O(n^3)$ complexity.

    Assume that all input data in the governing PDE are given in the low-rank tensor form. 
    The convenient tensor format for these problems is a canonical tensor representation of both 
    the governing operator, and of the initial guess as well as of the right hand side. 
    The commonly used numerical techniques 
  are based on certain iterative schemes that include at each iterative step multiple matrix-vector and  
  vector-vector algebraic operations
  each of them enlarges the tensor rank of the output in the additive or multiplicative way.
  It turns out that in common practice the most computationally intensive step 
  in the rank-structured algorithms is the adaptive rank truncation,  
  which makes the rank truncation  procedure ubiquitous.

  We notice that in PDE based mathematical models the total numerical complexity of the particular 
  computational scheme, i.e. 
  the overall cost of the  rank truncation procedure  is determined by the multiple of 
   the number of calls to the rank truncation algorithm (merely the number of iterations) 
  and the cost of a single RHOSVD transform (mainly determined by the rank parameter of the input tensor). 
  In turn, both complexity characteristics depend on the quality of the rank-structured preconditioner so that 
  optimization of the whole solution process is can be achieved by the trade-off 
  between Kronecker rank of the preconditioner   and the complexity of its implementation. 
  
    In the course of preconditioned iterations, 
   the tensor ranks of the governing operator, the preconditioner and the iterand are multiplied, 
   and therefore a robust rank reduction is mandatory procedure for 
   such techniques applied to iterative solution of elliptic and 
   pseudo-differential equations in the rank-structured tensor format.
   
     In particular, the RHOSVD was applied to the numerical solution of 
   PDE constrained (including the case of fractional operators) 
   optimal control problems \cite{HKKS:18,SKKS:21}, 
      where the complexity of the order $O(n R\log n  )$ was demonstrated. 
  In the case of higher dimensions the rank reduction in the canonical format can be performed 
 directly (i.e., without intermediate use of the Tucker approximation) by using 
 the cascading  ALS iteration in the CP format, see \cite{KhSch:11} concerning the tensor-structured 
 solution of the stochastic/parametric PDEs.

\section {Conclusions} \label{sec:conclusion}

We discuss theoretical and computational aspects of the RHOSVD served for 
approximation of tensors in low-rank Tucker/canonical formats,
and show that this rank reduction technique  is the principal ingredient in 
tensor-based computations for real-life problems in scientific computing and data modeling.  
We recall rank reduction scheme for the canonical input tensors based
on RHOSVD and subsequent Tucker-to-canonical transform. We present   
the detailed error analysis of low rank RHOSVD approximation to the canonical tensors 
(possibly with large input rank), and provide the proof on the uniform bound for the relative approximation error.
 
 We recall that the first example on application of the RHOSVD  was the  rank-structured 
 computation of the 3D convolution transform with the nonlocal Newton kernel in $\mathbb{R}^3$, 
 which is the basic  operation in the Hartree-Fock calculations. 
 
% 
%  We overview the sinc based analytic tensor approximation methods applied to the important class of 
% radial functions in $\mathbb{R}^d$,
% arising in a wide range of applications. We then discuss how the discretized radial functions can 
% be represented in the RS  tensor format. 

The RHOSVD is the basic tools for utilizing  the multilinear algebra in RS tensor format, which employs
the sinc-analytic tensor approximation methods applied to the important class of radial functions in $\mathbb{R}^d$. 
This enables efficient rank decompositions of
tensors generated by functions with multiple local cusps or singularities by separating 
their short- and long-range parts.
As an example, we construct the RS tensor representation of the discretized Slater 
function $e^{-\lambda \|x\|}$, $x\in\mathbb{R}^d $.
We then describe the RS tensor approximation  to various Green's kernels 
obtained by combination of this function with other potentials, in particular, with the Newton kernel 
  providing the Yukawa potential.  In this way, we introduce the concept of reproducing radial 
functions which pave the way for efficient RS tensor decomposition 
applied to a wide range of function-related multidimensional data  by combining the multilinear 
algebra in RS tensor format with the RHOSDV rank reduction techniques.

Our next example is related to application of RHOSVD to low-rank tensor interpolation of scattered data.
Our numerical tests demonstrate the efficiency of this approach on the example 
of multi-Slater interpolant in the case of many measurement points. 
We apply the RHOSVD to the data generated via random or function modulated amplitudes of samples 
and demonstrate numerically that for both cases the rank of the long-range
part remains small and depends weakly on the number of samples.

Finally, we notice that the described RHOSVD algorithms have proven their efficiency 
in a number of recent applications, in particular, 
in rank reduction for the  tensor-structured iterative solvers for PDE constraint optimal control problems 
(including fractional control), in construction of the range-separated tensor representations for calculation 
of the electrostatic potentials of many-particle systems (arising in protein modeling), and for 
numerical analysis of large scattered data in $\mathbb{R}^d$.

\begin{footnotesize}

\end{footnotesize}

\end{document}